\documentclass[submission]{eptcs}
\providecommand{\event}{ACT 2023}

\usepackage[british]{babel}
\usepackage[utf8]{inputenc}
\usepackage[T1]{fontenc}

\hypersetup{colorlinks, urlcolor=RoyalBlue, linkcolor=Red!80, citecolor=Green} % Needed to display links
\usepackage[hyphenbreaks]{breakurl}

\usepackage{graphicx} % Needed to include images
\usepackage[dvipsnames]{xcolor} % Needed to use color names in hyperref options

  \usepackage{tikz} % TikZ
  \usetikzlibrary{
    cd, % to make easy commutative diagrams
    petri, % To draw petri nets
    backgrounds, % To define image backgrounds
    arrows, % To use and define further arrow tips
    positioning, % To use expressions like "right = 1 of 1"
    decorations.markings, % Needed to define oriented wiring diagrams
    decorations.pathmorphing,
    calc,  % Needed to define oriented wiring diagrams
    fit, % Needed to compose wiring diagrams
    shapes.multipart
  }

\tikzset{curve/.style={settings={#1},to path={(\tikztostart)
    .. controls ($(\tikztostart)!\pv{pos}!(\tikztotarget)!\pv{height}!270:(\tikztotarget)$)
    and ($(\tikztostart)!1-\pv{pos}!(\tikztotarget)!\pv{height}!270:(\tikztotarget)$)
    .. (\tikztotarget)\tikztonodes}},
    settings/.code={\tikzset{quiver/.cd,#1}
        \def\pv##1{\pgfkeysvalueof{/tikz/quiver/##1}}},
    quiver/.cd,pos/.initial=0.35,height/.initial=0}

\tikzset{tail reversed/.code={\pgfsetarrowsstart{tikzcd to}}}
\tikzset{2tail/.code={\pgfsetarrowsstart{Implies[reversed]}}}
\tikzset{2tail reversed/.code={\pgfsetarrowsstart{Implies}}}
\tikzset{no body/.style={/tikz/dash pattern=on 0 off 1mm}}

\usepackage{mathtools} % Basic math capabilities
\usepackage{amssymb} % Extra mathematical symbols (loads amsfonts)
\usepackage{dsfont}
\usepackage{mathrsfs}

\usepackage{amsthm} % Theorem environments
\usepackage{stmaryrd} % Some maths symbols for logic and computer science

\usepackage{relsize} % Additional relative sizes for fonts
\usepackage{microtype} % Improves appearance of writing
\usepackage{multicol} % Multi-column environments
\usepackage{csquotes} % Environments for quotes
\usepackage{xspace}
\usepackage{enumitem}
\usepackage{fontawesome} % Popular fontawesome symbols

\usepackage{tabularx, cellspace} % Needed to typeset the tables for listing symbols
\usepackage[capitalize]{cleveref}

\usepackage{datetime}

\newcounter{theoremUnified} % Unified coutner for all theorem environments
\numberwithin{theoremUnified}{section} % Numbering within sections
\numberwithin{theoremUnified}{section} % Equations are also numbered within sections

\newtheoremstyle{plainStyle} % Plain theorem style
	{2mm} % Space above
	{2mm} % Space below
	{} % Body font
	{} % Indent amount
	{\bfseries} % Theorem head font
	{.} % Punctuation after theorem head
	{.5em} % Space after theorem head
	{} % Theorem head spec (can be left empty, meaning `normal')

\newtheoremstyle{italicStyle} % Italic theorem style
	{2mm} % Space above
	{2mm} % Space below
	{\itshape} % Body font
	{} % Indent amount
	{\bfseries} % Theorem head font
	{.} % Punctuation after theorem head
	{.5em} % Space after theorem head
	{} % Theorem head spec (can be left empty, meaning `normal')

\theoremstyle{plainStyle}
	{\rmfamily}{\rmfamily}
	\newtheorem{definition}{Definition}{\rmfamily}{\rmfamily}
	\newtheorem{remark}{Remark}{\rmfamily}{\rmfamily}
	\newtheorem{example}{Example}{\rmfamily}{\rmfamily}
	{\rmfamily}{\rmfamily}

\theoremstyle{italicStyle}
	\newtheorem{proposition}{Proposition}{\rmfamily}{\rmfamily}
	{\rmfamily}{\rmfamily}
	\newtheorem{corollary}{Corollary}{\rmfamily}{\rmfamily}

\newcommand\eqdef{\coloneqq}
\newcommand\nbd{\nobreakdash-\hspace{0pt}}
\newcommand\idd[1]{\mathrm{id}_{#1}}

\newcommand\invrs[1]{#1^{-1}}

\newcommand\incl{\hookrightarrow}

\newcommand\restr[2]{{#1}{\raisebox{0pt}{$|_{#2}$}}}
\newcommand\powerset[1]{\mathscr{P}{#1}}

\newcommand\slice[2]{{#1}/{\raisebox{-2pt}{$#2$}}}
\newcommand\laxslice[2]{{#1}\nnearrow\!{\raisebox{-2pt}{$#2$}}}
\newcommand\pointed[1]{#1_{\bullet}}

\newcommand\indic[1]{\chi_{#1}}

\newcommand\opp[1]{{#1}^\mathrm{op}}

\newcommand\lcat[1]{\mathbf{#1}}
\newcommand\cat[1]{#1}

\newcommand\fun[1]{\mathsf{#1}}

\newcommand\homset[3]{\mathrm{Hom}_{#1}(#2, #3)}
 
\newcommand\wkarr{{\vec{\cat{I}}}}
\newcommand\catpos{\lcat{Pos}}
\newcommand\catset{\lcat{Set}}
\newcommand\catcat{\lcat{Cat}}
\newcommand\lcatcat{\mathscr{C}\!\textit{at}}

\newcommand\catrel{\lcat{Rel}}
\newcommand\catgrph{\lcat{Graph}}
\newcommand\catopengrph{\lcat{OpenGraph}}

\newcommand\pastext[1]{\lcat{Past}(#1)}
\newcommand\extfun[2]{\fun{Ext}^{#1}_{#2}}

\newcommand\posref[1]{\| #1 \|}
\newcommand\pararr[2]{\mathrm{Par}(\slice{#1}{#2})}

\newcommand\dhom[3]{\pi_{#1}(\slice{#2}{#3})}

\definecolor{myblue}{RGB}{92,92,214}
\definecolor{myred}{RGB}{214,92,92}
\definecolor{mymagenta}{RGB}{214,92,214}
\definecolor{mygreen}{RGB}{36,143,36}
\newcommand\blue[1]{\textcolor{myblue}{#1}}
\newcommand\red[1]{\textcolor{myred}{#1}}
\newcommand\magenta[1]{\textcolor{mymagenta}{#1}}

\newcommand{\Cp}{\fatsemi} % Morphism composition in diagrammatic order

\newcommand{\CategoryC}{\mathit{C}}
\newcommand{\CategoryD}{\mathit{D}}

\newcommand{\WTerm}{\mathds{1}}
\newcommand{\Term}{\mathbf{1}}

\newcommand{\Set}{\mathbf{Set}} % Category of sets and functions
\newcommand{\TensorUnit}{I} % Monoidal tensor unit

\begin{document}
    \title{Obstructions to Compositionality}

    \author{        
        Caterina Puca
            %\email{OrcID here}
            \institute{Quantinuum\textsuperscript{*}}
            \email{caterina.puca@quantinuum.com}
        \and
        Amar Hadzihasanovic
            %\email{OrcID here}
            \institute{$^1$ Quantinuum\textsuperscript{*} \\$^2$ Tallinn University of Technology}%\footref{note: Quantinuum Address}}
            \email{amar.hadzihasanovic@quantinuum.com}
        \and
        Fabrizio Genovese
            %\email{OrcID here}
            \institute{20squares}
            \email{noreply@20squares.xyz}
        \and
        Bob Coecke
            %\email{OrcID here}
            \institute{Quantinuum\footnote{17 Beaumont Street, Oxford OX1 2NA, United Kingdom}}
            \email{bob.coecke@quantinuum.com}
    }

    \def\titlerunning{Obstructions to Compositionality}
    \def\authorrunning{Puca, Hadzihasanovic, Genovese, Coecke}

    \maketitle
    \begin{abstract}

Compositionality is at the heart of computer science and several other areas of applied category theory such as computational linguistics, categorical quantum mechanics, interpretable AI, dynamical systems, compositional game theory, and Petri nets. 
However, the meaning of the term seems to vary across the many different applications.
This work contributes to understanding, and in particular qualifying, different kinds of compositionality.

Formally, we introduce invariants of categories that we call zeroth and first homotopy posets, generalising in a precise sense the $\pi_0$ and $\pi_1$ of a groupoid.
These posets can be used to obtain a qualitative description of how far an object is from being terminal and a morphism is from being iso.
In the context of applied category theory, this formal machinery gives us a way to qualitatively describe the ``failures of compositionality'', seen as failures of certain (op)lax functors to be strong, by classifying obstructions to the (op)laxators being isomorphisms. 

Failure of compositionality, for example for the interpretation of a categorical syntax in a semantic universe, can both be a bad thing and a good thing, which we illustrate by respective examples in graph theory and quantum theory.

    \end{abstract}
    \paragraph*{\bf Acknowledgements}
        A.H.\ was supported by the ESF funded Estonian IT Academy research measure (project 2014-2020.4.05.19-0001) and by the Estonian Research Council grant PSG764.
        We thank Sean Tull and Robin Lorenz for helpful comments on an earlier draft.

    \section*{Introduction}\label{sec: introduction}
\emph{Compositionality} is probably the most relevant principle in applied category theory (ACT) research.
While there is no unified definition \cite{Werning_Hinzen_Machery_2012, ghani2018compositional, coecke2021compositionality}, it refers, broadly speaking, to certain forms of relation between properties, behaviours, or observations of a composite system on one hand, and those of its components on the other.
A common concern, in this context, is whether it is possible to derive properties of the whole from properties of its parts, and vice versa.
In some cases, both directions are viable and inverse to each other, in which case a property is ``fully compositional''.
More frequently, only one direction is viable.

The need to formally quantify and/or qualify compositionality has been widely discussed in the ACT community at least since 2018 \cite{Genovese2018mod}, as researchers became increasingly aware of various ``failures of compositionality'', and wished to classify them beyond a simple yes-or-no statement.

Let us be more precise.
Much research in ACT has been devoted to the study of \emph{open systems}, that is, entities with open interfaces that can be composed with other entities of the same kind.
This approach has been pervasive, and has been applied in the study of \emph{categorical quantum mechanics} \cite{abramsky2009categorical}, \emph{natural language} \cite{coecke2021mathematics},  \emph{dynamical systems} \cite{Fong_Spivak_2019}, \emph{Petri nets} \cite{baez2021categories}, \emph{game theory} \cite{ghani2018compositional} and many other subjects.
When studying open systems, it is not rare to define functors mapping a ``theory of boxes''  --- in the form of a monoidal category or bicategory --- where the composition rules of the systems are defined, to a certain ``semantic universe'' of properties or behaviours of the systems.
The properties of these functors reflect how well the information that they capture adheres to the composition rules: a \emph{lax} functor $\fun{P}$, with structural \emph{laxator} morphisms in the direction $\fun{P}f \Cp \fun{P}g \to \fun{P}(f \Cp g)$, means that one can derive information on the whole system from information on its components; an \emph{oplax} functor, with structural morphisms in the direction $\fun{P}(f \Cp g) \to \fun{P}f \Cp \fun{P}g$, means that one can derive information on the components from information on the whole; while a \emph{strong} functor means that the information on components and the information on the whole completely determine each other.

For example, the functor sending \emph{open graphs} to their \emph{reachability relation} (see \autoref{subsec: open graphs}) is lax, which tells us that the reachability relation of a composition of open graphs can be strictly bigger than the composition of the reachability relations defined on its parts.
This is considered undesirable from a computational viewpoint, as it means that one cannot reconstruct the reachability of a graph by separately computing the reachability of its components.

On the other hand, in ``Schr\"odinger compositionality'' (covered in \autoref{subsec: schrodinger compositionality}), quantum-mechanical behaviour arises from the laxity of the functor mapping each object to its set of states.
This laxity implies that not all quantum states are separable, which is desirable, as it unlocks the use of \emph{entanglement} as a resource unavailable in classical mechanics.

In both cases, laxity represents a ``failure of compositionality'' which has both practical and foundational importance: the ``gap'' between a lax and a strong functor represents the gap between what we can compute compositionally with a ``divide-and-conquer'' strategy and what we cannot, or the gap between a classical and non-classical theory of processes.
In this light, the question: \emph{how can we qualify (failures of) compositionality?} becomes the question: \emph{how far is a lax functor from being strong?}\footnote{We will focus on lax functors in our discussion, but everything can be dualised to oplax functors.}
In this paper, we attempt to give a structured answer to the question.
Our chain of reasoning is the following.
\begin{definition}
    A lax functor is strong when all the components of its laxators are isomorphisms.
\end{definition}
Thus, we can think of reducing our question to the more general one: \emph{how far is a morphism from being an isomorphism?}\footnote{This approach, and the fact that it could be investigated with homotopical methods, was first suggested to us by Jules Hedges.}
Let us use the following, well-known characterisation of isomorphisms.
\begin{proposition}\label{prop: iso if terminal in slice}
    A morphism $f\colon X \to Y$ in a category $\CategoryC$ is an isomorphism if and only if it is terminal as an object of the slice category $\slice{\CategoryC}{Y}$.
\end{proposition}
This allows us to reduce further to the question: \emph{how far is an object from being terminal?}
Terminality can be split into the following pair of properties.
\begin{definition}\label{def: weak and subterminal}
    An object $\WTerm$ in a category $\CategoryC$ is
    \begin{itemize}
        \item \emph{weak terminal} if, for all objects $X$ of $\CategoryC$, there exists a morphism $X \to \WTerm$;
        \item \emph{subterminal} if, for all parallel pairs of morphisms $f, g\colon X \to \WTerm$, we have $f = g$.
    \end{itemize}
\end{definition}
Hence, to describe how far $\WTerm$ is from being terminal, we can separately describe how far $\WTerm$ is from being weak terminal and subterminal, respectively.

Following this chain of reasoning, we focus on classifying \emph{obstructions to weak terminality and subterminality} for objects in arbitrary categories.
Surprisingly, it turns out that there exists a natural way of associating certain \emph{pointed posets} to a pointed category (category with a chosen object), which we call the \emph{zeroth} and \emph{first homotopy poset}, because in a precise sense they generalise the $\pi_0$ and $\pi_1$ of a pointed groupoid seen as a homotopy 1-type.
This opens up the possibility of an \emph{invariant-based} approach to the formal study of compositionality: the homotopy posets contain no information that is not already in the functors and categories, but put it in a form which may be more tractable and intelligible.

In \autoref{sec: homotopy poset}, we give the definitions of homotopy posets and state their basic properties, demonstrating in which sense they answer our question about terminal objects.
In \autoref{sec: obstructions}, going backwards in our chain of reasoning, we apply them to the study of obstructions to morphisms being iso.
Finally, in \autoref{sec: qualifying}, we sketch through a couple of simple examples how our framework can be applied to the study of failures of compositionality, seen as failures of certain (op)lax functors to be strong.
Some particularly involved proofs are collected in the Appendix; we refer to the extended version \cite{puca2023obstructions} for other proofs and further details.

    \section{Homotopy posets}\label{sec: homotopy poset}
To begin, we focus on obstructions to weak terminality.
Having fixed a category $\CategoryC$, we interpret objects of a category $\CategoryC$ as points, and morphisms between them as paths. 
From this point of view, a weak terminal object is an object that is always reachable from any generic object $x$ in $\CategoryC$.

Intuitively, we can fix a ``weak terminal object candidate''\footnote{
    In this paper, we will use $\WTerm$ to denote ``terminal object candidates'', that is, objects for which we want to investigate how far they are from being terminal. For an object that we know or presume to be terminal, we will instead use the notation $\Term$.
} $\WTerm$ and consider any object $x$ such that there is \emph{no} morphism $x \to \WTerm$ as an \emph{obstruction to weak terminality}.
Moreover:
\begin{itemize}
    \item If $x, y$ are obstructions for $\WTerm$, and there are morphisms $x \to y$ and $y \to x$, we regard them as equivalent: if there were a morphism $x \to \WTerm$ there would be a morphism $y \to \WTerm$, and vice versa.
    \item If $x,y$ are obstructions for $\WTerm$ and there is a morphism $x \to y$, then we regard $x$ as a ``more fundamental obstruction than $y$''. 
    This is because, if there were a morphism $y \to \WTerm$, we would automatically obtain a morphism $x \to \WTerm$ by composition (one can ``go from $x$ to $y$ and then to $\WTerm$''), while the opposite is not true.
\end{itemize}
We will devote this section to making this intuition formal.
\begin{definition}[Poset reflection] \label{def: poset reflection}
    Let $\catpos$ be the large\footnote{We will denote categories in \textit{italics} and large categories in \textbf{bold}. Note that in our constructions, what matters is only the \emph{relative} size: a construction which associates a poset to a category can be applied to a large category, producing a large poset.}
    category of posets and order\nbd preserving maps.
    There is a full and faithful functor $\imath\colon \catpos \incl \catcat$, whose image consists of the categories that are
    \begin{itemize}
        \item \emph{thin} (each hom\nbd set contains at most one morphism), and
        \item \emph{skeletal} (every isomorphism is an automorphism).
    \end{itemize}
    The \emph{poset reflection} $\posref{\CategoryC}$ of a category $\CategoryC$ is its image under the left adjoint $\posref{-}\colon \catcat \to \catpos$ to $\imath$:
    \begin{itemize}
        \item the elements of $\posref{\CategoryC}$ are equivalence classes $\posref{x}$ of objects $x$ of $\CategoryC$, where $\posref{x} = \posref{y}$ if and only if there exist morphisms $x \to y$ and $y \to x$ in $\CategoryC$, and
        \item $\posref{x} \leq \posref{y}$ if and only if there exists a morphism $x \to y$ in $\CategoryC$.
    \end{itemize}
\end{definition}
\begin{proposition}\label{prop: weak terminal is greatest in posref}
    Let $\CategoryC$ be a category and $\WTerm$ an object in $\CategoryC$.
    The following are equivalent:
    \begin{enumerate}[label=(\alph*)]
        \item $\WTerm$ is a weak terminal (respectively, initial) object in $\CategoryC$;
        \item $\posref{\WTerm}$ is the greatest (respectively, least) element of $\posref{\CategoryC}$.
    \end{enumerate}
\end{proposition}
\begin{definition}[Arrow category]\label{def: arrow category}
    Let $\wkarr$ be the ``walking arrow'' category, that is, the free category on the graph
    \[\begin{tikzcd}[sep=scriptsize]
        0 && 1
        \arrow["a", from=1-1, to=1-3]
    \end{tikzcd}.\]
    The \emph{arrow category} of a category $\CategoryC$ is the functor category $\CategoryC^\wkarr$.
    Explicitly, the objects of $\CategoryC^\wkarr$ are morphisms of $\CategoryC$, while morphisms of $\CategoryC^\wkarr$ are commutative squares in $\CategoryC$.
    There are functors $\mathrm{dom}$, $\mathrm{cod}\colon \CategoryC^\wkarr \to \CategoryC$ which, given a morphism $(h_0, h_1)$, return $h_0$, respectively, $h_1$.
\end{definition}
\begin{definition}[Category of pointed objects]\label{def: pointed objects category}
    Let $\CategoryC$ be a category with a chosen terminal object $\Term$.
    A \emph{pointed object} $(x, v)$ of $\CategoryC$ is an object $x$ of $\CategoryC$ together with a morphism $v\colon \Term \to x$, called its \emph{basepoint}.
    The \emph{category of pointed objects} of $\CategoryC$ --- denoted by $\pointed{\CategoryC}$ --- is the coslice category $\slice{\Term}{\CategoryC}$.
\end{definition}
\begin{proposition}[Functoriality of arrow and pointed objects categories]\label{prop: functoriality of arrow and pointed cats}
    Let $\fun{F}\colon \CategoryC \to \CategoryD$ be a functor.
    Then $\fun{F}$ lifts to a functor $\fun{F}^\wkarr\colon \CategoryC^\wkarr \to \CategoryD^\wkarr$
    using the pointwise action of $\fun{F}$ on $\CategoryC$. 
    
    If moreover $\CategoryC$ and $\CategoryD$ have a chosen terminal object, and if $\fun{F}$ preserves it, then it also lifts to a functor $\pointed{\fun{F}}\colon \pointed{\CategoryC} \to \pointed{\CategoryD}$ sending a pointed object $(x, v)$ of $\CategoryC$ to $(\fun{F}x, \fun{F}v)$, a pointed object of $\CategoryD$.
\end{proposition}
\begin{definition}[Quotient of an object by a morphism]\label{def: quotient by a morphism}
    Let $\CategoryC$ be a category with chosen pushouts and a terminal object $\Term$.
    Given a morphism $f\colon x \to y$, the \emph{quotient of $y$ by $f$} is the pushout
    \begin{equation*}
        \begin{tikzcd}[sep=scriptsize]
            x && \Term \\
            \\
            y && {y\sslash f}
            \arrow["{!}", from=1-1, to=1-3]
            \arrow["f"', from=1-1, to=3-1]
            \arrow[from=3-1, to=3-3]
            \arrow["{[x]}", from=1-3, to=3-3]
            \arrow["\lrcorner"{anchor=center, pos=0.125, rotate=180}, draw=none, from=3-3, to=1-1]
        \end{tikzcd}
    \end{equation*}
    where $!\colon x \to \Term$ is the unique morphism from $x$ to the terminal object.
\end{definition}
\begin{proposition}[Functoriality of the quotient]\label{prop: functoriality of the quotient}
    If $\CategoryC$ has chosen pushouts and a terminal object $\Term$, then for each morphism $f\colon x \to y$ in $\CategoryC$ \autoref{def: quotient by a morphism} determines a pointed object $\fun{Q}(f) \eqdef (y \sslash f, [x])$ of $\CategoryC$. This extends to a functor $\fun{Q}\colon \CategoryC^\wkarr \to \pointed{\CategoryC}$.
    If both $\CategoryC$ and $\CategoryD$ have chosen pushouts and a chosen terminal object $\Term$, and if $\fun{F}$ preserves them, then $\fun{F}$ induces a commutative square of functors
    \begin{equation*}
        \begin{tikzcd}[sep=scriptsize]
            \CategoryC^\wkarr && \pointed{\CategoryC} \\
            \\
            \CategoryD^\wkarr && \pointed{\CategoryD}.
            \arrow["\fun{Q}", from=1-1, to=1-3]
            \arrow["\fun{F}^\wkarr"', from=1-1, to=3-1]
            \arrow["\fun{Q}"', from=3-1, to=3-3]
            \arrow["\pointed{\fun{F}}", from=1-3, to=3-3]
        \end{tikzcd}
    \end{equation*}
\end{proposition}
The categories $\catcat$ and $\catpos$ have all limits and colimits, so in particular they have pushouts and a terminal object. The poset reflection functor $\posref{-}\colon \catcat \to \catpos$ sends the terminal category to the terminal poset, and preserves pushouts, since it is a left adjoint.
The preservation can be made strict with respect to a choice on both sides.
We are in the conditions of \autoref{prop: functoriality of the quotient}: there is a commutative square
\begin{equation} \label{eq: quotient and posref}
    \begin{tikzcd}[sep=scriptsize]
        \catcat^\wkarr && \pointed{\catcat} \\
        \\
        \catpos^\wkarr && \pointed{\catpos}.
        \arrow["\fun{Q}", from=1-1, to=1-3]
        \arrow["\posref{-}^\wkarr"', from=1-1, to=3-1]
        \arrow["\fun{Q}", from=3-1, to=3-3]
        \arrow["\pointed{\posref{-}}", from=1-3, to=3-3]
    \end{tikzcd}
\end{equation}
We are now ready to define the object of interest of this section.
\begin{definition}[Zeroth homotopy poset] \label{def: 0th-directed homotopy poset}
    Let $\CategoryC$ be a category and $x$ an object in $\CategoryC$.
    The \emph{zeroth homotopy poset of $\CategoryC$ over $x$} is the pointed poset
    \begin{equation*}
        (\dhom{0}{\CategoryC}{x}, \; [x])
    \end{equation*}
    obtained by applying the functor $\catcat^\wkarr \to \pointed{\catpos}$ from \autoref{eq: quotient and posref} to the slice projection functor
    \begin{equation*}
        \mathrm{dom}\colon \slice{\CategoryC}{x} \to \CategoryC.
    \end{equation*}
\end{definition}
Let us unravel the definition of $\dhom{0}{\cat{C}}{x}$ to a more explicit form.
We start from the projection functor $\mathrm{dom}\colon \slice{\cat{C}}{x} \to \cat{C}$.
    To this we may either apply $\fun{Q}$ or $\posref{-}^\wkarr$.
    Since quotients in $\catpos$ are simpler to compute than quotients in $\catcat$, we apply poset reflection first, which gives us an order-preserving map
    \begin{equation*}
        \posref{\mathrm{dom}}\colon \posref{\slice{\cat{C}}{x}} \to \posref{\cat{C}}.
    \end{equation*}
    Unravelling the explicit definition of poset reflection for $\slice{\cat{C}}{x}$, we see that:
    \begin{itemize}
        \item an element of $\posref{\slice{\cat{C}}{x}}$ is an equivalence class $\posref{f\colon y \to x}$ of morphisms of $\cat{C}$ with codomain $x$, where $\posref{f} = \posref{g}$ if and only if $f$ factors through $g$ and $g$ factors through $f$, and
        \item $\posref{f} \leq \posref{g}$ if and only if $f$ factors through $g$.
    \end{itemize}
    The map $\posref{\mathrm{dom}}$ sends $\posref{f}$ to $\posref{\mathrm{dom}\, f}$.
    The image of $\posref{\mathrm{dom}}$ is then the set
    \begin{equation*}
        \{ \posref{y} \mid \text{there exists a morphism $f\colon y \to x$ in $\cat{C}$} \},
    \end{equation*}
    which is, equivalently, the lower set of $\posref{x}$ in $\posref{\cat{C}}$.

    Applying $\fun{Q}\colon \catpos^\wkarr \to \pointed{\catpos}$ to this map produces the quotient of $\posref{\cat{C}}$ with all elements of this set identified, pointed with the element resulting from their identification, which we denote by $[x]$.
    Hence, an element of $\dhom{0}{\CategoryC}{x}$ is either $[x]$, or it is $\posref{y}$ for some object $y$ such that there exists no morphism $f\colon y \to x$ in $\CategoryC$.
    The order relation is defined as follows, by case distinction:
    \begin{itemize}
        \item $[x] \leq [x]$ trivially;
        \item $[x] \leq \posref{y}$ if and only if there exists a span $(x \xleftarrow{f} z \xrightarrow{g} y)$ in $\CategoryC$;
        \item it is never the case that $\posref{y} \leq [x]$;
        \item $\posref{y} \leq \posref{z}$ if and only if there exists a morphism $f\colon y \to z$ in $\CategoryC$.
    \end{itemize}
    Notice that $[x]$ is always minimal in $\dhom{0}{\CategoryC}{x}$.

The partial order on $\dhom{0}{\CategoryC}{x}$ ranks obstructions to weak terminality by ``size'': if we removed an obstruction $\posref{y}$, adding a morphism $y \to x$, we would also have to remove all the ``smaller'' obstructions $\posref{z} \leq \posref{y}$.
The minimal element $[x]$ represents the ``non-obstructions'':
\begin{proposition} \label{prop:dhom0_trivial_when_weak_terminal}
Let $\cat{C}$ be a category and $x$ an object in $\cat{C}$.
The following are equivalent:
\begin{enumerate}[label=(\alph*)]
    \item $\dhom{0}{\cat{C}}{x} = \{[x]\}$;
    \item $x$ is a weak terminal object in $\cat{C}$.
\end{enumerate}
\end{proposition}

The notation and terminology is suggestive of the $\pi_0$ of a pointed topological space or groupoid, that is, its set of connected components, pointed with the connected component of the basepoint. 
The following result shows that, indeed, the notions coincide when $\cat{C}$ happens to be a groupoid.

\begin{proposition}[$\dhom{0}{\cat{G}}{x}$ for a groupoid]\label{prop: dhom0 is pi0 for groupoids}
    Let $\cat{G}$ be a groupoid and $x$ an object in $\cat{G}$.
    Then
    \begin{enumerate}
        \item $\dhom{0}{\cat{G}}{x}$ is a ``set'', that is, a discrete poset, and
        \item as a pointed set, it is isomorphic to the set $\pi_0(\cat{G})$ of connected components of $\cat{G}$, pointed with the connected component of $x$.
    \end{enumerate}
\end{proposition}

    Now, we investigate obstructions to \emph{subterminality}.
Our main strategy will be to recast subterminality in a way that allows us to leverage \autoref{def: 0th-directed homotopy poset}.
We know that an object $\WTerm$ fails to be subterminal when, for an object $x$, the arrow $x \to \WTerm$ is not unique.
As such, we will describe obstructions to subterminality as pairs of parallel, unequal arrows.
\begin{definition}[Category of parallel arrows over an object]\label{def: category of parallel arrows}
    Let $\CategoryC$ be a category and $x$ an object in $\CategoryC$.
    The \emph{category of parallel arrows in $\CategoryC$ over $x$} is the category $\pararr{\CategoryC}{x}$ where:
    \begin{itemize}
        \item Objects are pairs of morphisms $(f_0, f_1\colon y \to x)$ with codomain $x$.
        \item A morphism from $(f_0, f_1\colon y \to x)$ to $(g_0, g_1\colon z \to x)$ is a morphism $h\colon y \to z$ such that $f_0 = h\Cp g_0$ and $f_1 = h\Cp g_1$.
    \end{itemize}
    This comes with a projection functor $\mathrm{dom}\colon \pararr{\CategoryC}{x} \to \CategoryC$ sending a parallel pair to its domain.
\end{definition}
\begin{proposition}\label{prop: subterminal as weak terminal parallel arrow}
    Let $\CategoryC$ be a category and $\WTerm$ an object in $\CategoryC$.
    The following are equivalent:
    \begin{enumerate}[label=(\alph*)]
        \item $\WTerm$ is subterminal in $\CategoryC$;
        \item $(\idd{\WTerm}, \idd{\WTerm})$ is a terminal object in $\pararr{\CategoryC}{\WTerm}$;
        \item $(\idd{\WTerm}, \idd{\WTerm})$ is a weak terminal object in $\pararr{\CategoryC}{\WTerm}$.
    \end{enumerate}
\end{proposition}
\autoref{prop: subterminal as weak terminal parallel arrow} allows us to reduce the study of obstructions to subterminality of an object $\WTerm$ in $\CategoryC$ to the study of obstructions to weak terminality of $(\idd{\WTerm}, \idd{\WTerm})$ in $\pararr{\CategoryC}{\WTerm}$. 
\begin{definition}[First homotopy poset] \label{def: 1st-directed homotopy poset}
    Let $\CategoryC$ be a category and $x$ an object in $\CategoryC$.
    The \emph{first homotopy poset of $\CategoryC$ over $x$} is the pointed poset
    \begin{equation*}
        (\dhom{1}{\CategoryC}{x}, \, [x]) \eqdef \left(\dhom{0}{\pararr{\CategoryC}{x}}{(\idd{x}, \idd{x})}, \, [(\idd{x}, \idd{x})]\right).
    \end{equation*}
\end{definition}
Putting together the description of the 0th homotopy poset, the definition of $\pararr{\CategoryC}{x}$ in \autoref{def: category of parallel arrows}, and \autoref{prop: subterminal as weak terminal parallel arrow}, we see that an element of $\dhom{1}{\CategoryC}{x}$ is either $[x]$, or $\posref{(f, g)}$ for some parallel pair of morphisms $f, g\colon y \to x$ in $\CategoryC$ with $f \neq g$.
    The order relation is defined as follows:
    \begin{itemize}
        \item $[x] \leq [x]$ trivially;
        \item $[x] \leq \posref{(f, g\colon y \to x)}$ if and only if there exists a morphism $h\colon z \to y$ in $\CategoryC$ equalising $(f, g)$, that is, satisfying $h\Cp f = h\Cp g$;
        \item it is never the case that $\posref{(f, g)} \leq [x]$;
        \item $\posref{(f, g\colon y \to x)} \leq \posref{(f', g'\colon y' \to x)}$ if and only if there exists a morphism $h\colon y \to y'$ such that $f = h\Cp f'$ and $g = h\Cp g'$ in $\CategoryC$.
    \end{itemize}
\begin{proposition} \label{prop:dhom1_trivial_when_subterminal}
Let $\cat{C}$ be a category and $x$ an object in $\cat{C}$.
The following are equivalent:
\begin{enumerate}[label=(\alph*)]
    \item $\dhom{1}{\cat{C}}{x} = \{[x]\}$;
    \item $x$ is subterminal in $\cat{C}$.
\end{enumerate}
\end{proposition}

\begin{corollary} \label{prop:dhoms_trivial_when_terminal}
Let $\cat{C}$ be a category and $x$ an object in $\cat{C}$.
The following are equivalent:
\begin{enumerate}[label=(\alph*)]
    \item $\dhom{0}{\cat{C}}{x} = \{[x]\}$ and $\dhom{1}{\cat{C}}{x} = \{[x]\}$,
    \item $x$ is a terminal object in $\cat{C}$.
\end{enumerate}
\end{corollary}

\begin{remark}\label{rem: pi1 of a groupoid}
    Recall that the (underlying set of the) fundamental group of a pointed topological space $(X, x)$ is defined by
    \begin{equation*}
        \pi_1(X, x) \eqdef \pi_0(\Omega(X, x), c_x)
    \end{equation*}
    where $\Omega(X, x)$ is the space of loops in $X$ based at $x$, and $c_x$ is the constant path at $x$.
    For a pointed groupoid, which may be seen as the fundamental groupoid of a pointed space, this reduces to the set of automorphisms of the object $x$, pointed with the identity automorphism.
    
    The definition of $\dhom{1}{\CategoryC}{x}$ is made in analogy with this, letting the category of parallel arrows over $x$ replace the space of loops based at $x$, and a pair of identity morphisms replace the constant path.
    The following result proves that, just like the zeroth homotopy poset, the first homotopy poset is a generalisation of its groupoidal analogue.
\end{remark}
\begin{proposition}[$\dhom{1}{\cat{G}}{x}$ for a groupoid]\label{prop: dhom1 is pi1 for groupoids}
    Let $\cat{G}$ be a groupoid and $x$ an object in $\cat{G}$.
    Then:
    \begin{enumerate}
        \item $\dhom{1}{\cat{G}}{x}$ is a ``set'', that is, a discrete poset, and
        \item as a pointed set, it is isomorphic to the underlying pointed set of the group $\pi_1(\cat{G}, x) = \homset{\cat{G}}{x}{x}$.
    \end{enumerate}
\end{proposition}

\begin{remark}
    We mention here that the field of \emph{directed algebraic topology} \cite{grandis2009directed, fajstrup2016directed} has also produced ``non-invertible'' versions of $\pi_1$, namely, the fundamental \emph{category} and \emph{monoids}, that apply to directed spaces.
    If applied to a category, these pick out ``tautologically'' the category itself and its monoids of endomorphisms.
    To our knowledge, there is no strong relation to our line of research.
\end{remark}

%The elements of $\dhom{0}{\CategoryC}{x}$ and $\dhom{1}{\CategoryC}{x}$ that are minimal in the complement of $\{ [x] \}$ are often of particular importance, for the following reason.
%
%
%\begin{proposition}[Existence of joins]\label{prop: If C has coproducts and slice has weak initials, then dhom0 has joins}
%Let $\cat{C}$ be a category, $x$ an object of $\cat{C}$, and $\kappa$ a cardinal.
%If $\cat{C}$ has $\kappa$\nbd small coproducts, then $\dhom{0}{\cat{C}}{x}$ and $\dhom{1}{\cat{C}}{x}$ have $\kappa$\nbd small joins.
%\end{proposition}
%
%This has the consequence that, in many cases, elements of the homotopy posets are describable as joins of smaller elements; in particular, minimal elements in the complement of $\{ [x] \}$.
%We will call these \emph{minimal obstructions}.

To conclude this section, we show in what way the homotopy posets are functorial in the pair $(\cat{C}, x)$ of a category and an object.

\begin{proposition}[Functoriality of the homotopy posets] \label{prop: Homotopy posets are functorial}
Let $\cat{C}$ be a category, $i \in \{0, 1\}$.
Then:
\begin{enumerate}
    \item the assignment $x \mapsto \dhom{i}{\cat{C}}{x}$ extends to a functor
        $\dhom{i}{\cat{C}}{-}\colon \cat{C} \to \pointed{\catpos}$;
    \item a functor $\fun{F}\colon \cat{C} \to \cat{D}$ induces a natural transformation 
        $\pi_i(\fun{F})\colon \dhom{i}{\cat{C}}{-} \Rightarrow \dhom{i}{\cat{D}}{\fun{F}-}.$
\end{enumerate}
Given another functor $\fun{G}\colon \cat{D} \to \cat{E}$,  this assignment satisfies
\begin{equation*}
    \pi_i(\fun{F}\Cp \fun{G}) = \pi_i(\fun{F}) \Cp \pi_i(\fun{G}), \quad \quad \pi_i(\idd{C}) = \idd{\dhom{i}{\cat{C}}{-}}.
\end{equation*}
\end{proposition}

A concise way of packaging this information is to say that $\pi_i$ defines a functor from $\catcat$ to the \emph{lax slice} $\laxslice{\lcatcat}{\pointed{\catpos}}$, where $\lcatcat$ is the ``huge'' category of possibly large categories.
The objects of the lax slice are pairs of a possibly large category $\lcat{C}$ and a functor $\lcat{C} \to \pointed{\catpos}$, and the morphisms are triangles of functors commuting up to a natural transformation.
Indeed, given $\fun{F}\colon \cat{C} \to \cat{D}$, we have a triangle
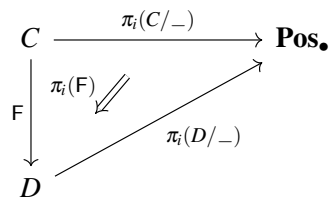
\begin{equation*}
\begin{tikzcd}
	{\cat{C}} &&& \pointed{\catpos} \\
	\\
	{\cat{D}}
	\arrow["{\fun{F}}"', from=1-1, to=3-1]
	\arrow["{\dhom{i}{\cat{D}}{-}}"', from=3-1, to=1-4]
	\arrow[""{name=0, anchor=center, inner sep=0}, "{\dhom{i}{\cat{C}}{-}}", from=1-1, to=1-4]
	\arrow["{\pi_i(\fun{F})}"', shorten <=17pt, shorten >=26pt, Rightarrow, from=0, to=3-1]
\end{tikzcd}
\end{equation*}
commuting up to the natural transformation $\pi_i(\fun{F})$.
\begin{remark}[Dual invariants] \label{rmk: Dual invariants}
As usual, all the constructions can be dualised to $\opp{\cat{C}}$.
This will replace the slice over an object and its domain opfibration with the slice under an object and its codomain fibration, producing invariants classifying obstructions to \emph{initiality} of the object.
\end{remark}

    \section{Obstructions to a morphism being iso} \label{sec: obstructions}

As remarked in the Introduction, one of our main motivations for introducing homotopy posets was measuring how far a generic morphism is from being iso.
Just as we could separate obstructions to terminality into obstructions to weak terminality and subterminality, we can separate obstructions to a morphism being iso into obstructions to a morphism being split epi and mono, respectively.
\begin{proposition}\label{prop: spit epi weak term mono subterm}
    Let $f\colon X \to Y$ be a morphism in a category $\CategoryC$. Then:
    \begin{itemize}
        \item $f$ is split epi in $\CategoryC$ if and only if $f$ is weak terminal in $\slice{\CategoryC}{Y}$,
        \item $f$ is mono in $\CategoryC$ if and only if $f$ is subterminal in $\slice{\CategoryC}{Y}$.
    \end{itemize}
\end{proposition}
\begin{corollary}\label{cor: spit epi iff dhom0 trivial mono iff dhom1 trivial}
    Let $f: X \to Y$ be a morphism in a category $\CategoryC$. Then:
    \begin{itemize}
        \item $f$ is split epi if and only if $\dhom{0}{(\slice{\cat{C}}{Y})}{f}$ is trivial;
        \item $f$ is mono if and only if $\dhom{1}{(\slice{\cat{C}}{Y})}{f}$ is trivial, and:
        \item $f$ is iso if and only if both $\dhom{0}{(\slice{\cat{C}}{Y})}{f}$ and $\dhom{1}{(\slice{\cat{C}}{Y})}{f}$ are trivial.
    \end{itemize}
\end{corollary}
Furthermore, when the homotopy posets associated to a morphism $f$ are not trivial, they give us precise information about why $f$ fails to be split epi and mono.

To make this more concrete, let us spell out precisely how to compute the invariants associated to a function between sets, where split epi (assuming choice) means \emph{surjective} and mono means \emph{injective}.
This amounts to calculating $\dhom{0}{(\slice{\Set}{Y})}{f}$ and $\dhom{1}{(\slice{\Set}{Y})}{f}$ for some function $f\colon X \to Y$. 
\begin{proposition}\label{prop: dhom0 for set/y}
    Let $f\colon X \to Y$ be a function between sets. $\posref{\slice{\catset}{Y}}$ is isomorphic, as a poset, to the power set $\powerset{Y}$, via the assignment $(S \subseteq Y) \mapsto \posref{\imath_S}$, where $\imath_S$ is the injective function including $S$ into $Y$.
    Through this bijection, $\posref{f}$ corresponds to the image $f(X)$ of $f$.
\end{proposition}
Using this correspondence and quotienting by the lower set of $f(X)$, which contains in particular $\varnothing$, we may identify $\dhom{0}{(\slice{\catset}{Y})}{f}$ with the subposet of $\powerset{Y}$ whose elements are either $\varnothing$ or subsets of $Y$ that contain at least one element $y \notin f(X)$.
The ``minimal obstructions'', that is, the minimal elements in the complement of the basepoint, are the singletons $\{y\}$ with $y \in Y \setminus f(X)$.
This poset is trivial if and only if $f(X) = Y$, that is, iff $f$ is surjective.
\begin{example}
    Let $f\colon \{0,1\} \to \{0,1,2,3\}$ be the function mapping $0 \mapsto 0$ and $1 \mapsto 1$. 
    The homotopy poset $\dhom{0}{(\slice{\catset}{\{0,1,2,3\}})}{f}$ has the following structure:
    \begin{equation*}
    \def\interval{1.75}
        \scalebox{0.75}
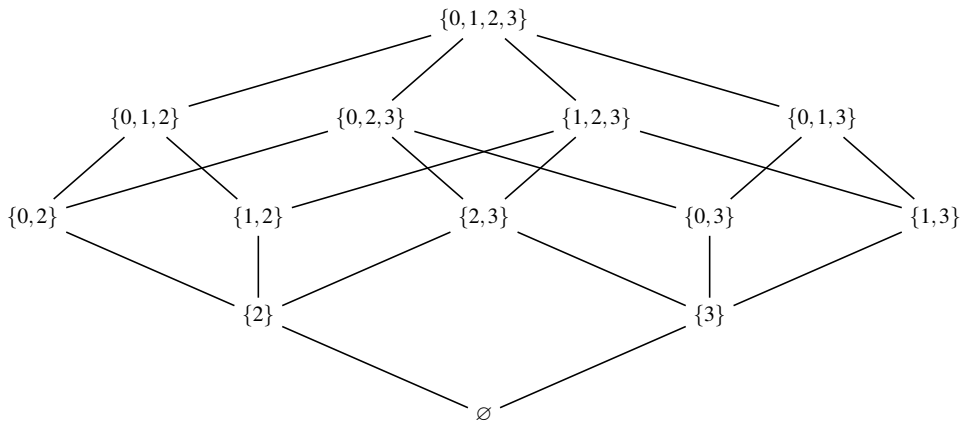
{
        \begin{tikzpicture}
            \node(4) at (0,4*\interval) {$\{0,1,2,3\}$};
            \node (3a) at (-6,3*\interval) {$\{0,1,2\}$};
            \node (3b) at (-2,3*\interval) {$\{0,2,3\}$};
            \node (3c) at (2,3*\interval) {$\{1,2,3\}$};
            \node (3d) at (6,3*\interval) {$\{0,1,3\}$};

            \node (2a) at (-8,2*\interval) {$\{0,2\}$};
            \node (2b) at (-4,2*\interval) {$\{1,2\}$};
            \node (2c) at (0,2*\interval) {$\{2,3\}$};
            \node (2d) at (4,2*\interval) {$\{0,3\}$};
            \node (2e) at (8,2*\interval) {$\{1,3\}$};

            \node (1a) at (-4,1*\interval) {$\{2\}$};
            \node (1b) at (4,1*\interval) {$\{3\}$};

        \node (0) at (0,0) {$\varnothing$};
            
            \draw[thick] (3a) -- (4);
            \draw[thick] (3b) -- (4);
            \draw[thick] (3c) -- (4);
            \draw[thick] (3d) -- (4);

            \draw[thick] (2a) -- (3a);
            \draw[thick] (2a) -- (3b);

            \draw[thick] (2b) -- (3a);
            \draw[thick] (2b) -- (3c);

            \draw[thick] (2c) -- (3b);
            \draw[thick] (2c) -- (3c);

            \draw[thick] (2d) -- (3b);
            \draw[thick] (2d) -- (3d);

            \draw[thick] (2e) -- (3c);
            \draw[thick] (2e) -- (3d);

            \draw[thick] (1a) -- (2a);
            \draw[thick] (1a) -- (2b);
            \draw[thick] (1a) -- (2c);

            \draw[thick] (1b) -- (2c);
            \draw[thick] (1b) -- (2d);
            \draw[thick] (1b) -- (2e);

            \draw[thick] (0) -- (1a);
            \draw[thick] (0) -- (1b);
        \end{tikzpicture}
        }
    \end{equation*}
    The minimal obstructions $\{2\}$ and $\{3\}$ are in bijection with the elements not in the image of $f$.
\end{example}
\begin{proposition}\label{prop: dhom1 for set/y}
    Let $X \times_f X$ be the pullback of $f$ along itself --- that is, the set $\{(x_0, x_1) \mid f(x_0) = f(x_1)\}$ --- and let $p_f\colon X \times_f X \to Y$ be the function $(x_0, x_1) \mapsto f(x_0) = f(x_1)$. Then:
    \begin{enumerate}
        \item $\posref{\pararr{(\slice{\catset}{Y})}{f}}$ is isomorphic to $\powerset{(X \times_f X)}$ via the assignment $(S \subseteq X \times_f X) \mapsto \posref{(\restr{p_0}{S}, \restr{p_1}{S})}$, where $\restr{p_i}{S}$ are the projections $X \times_f X \to Y$, restricted to $S$, seen as morphisms $\restr{p_f}{S} \to f$ in $\posref{\pararr{(\slice{\catset}{Y})}{f}}$;
        \item through this bijection, $\posref{(\idd{f}, \idd{f})}$ is identified with the diagonal $\Delta X$.
    \end{enumerate}
\end{proposition}
Using this correspondence, we may identify $\dhom{1}{\catset}{X}$ with the subposet of $\powerset{(X \times_f X)}$ whose elements are either $\varnothing$, or contain at least one pair $(x_0, x_1)$ such that $x_0 \neq x_1$.
This poset is trivial if and only if $f$ is injective. 
Notice that the minimal obstructions to injectiveness of $f$ are in bijection with pairs $(x_0, x_1)$ where $x_0 \neq x_1$ but $f(x_0) = f(x_1)$.
\begin{example}
    Let $f: \{0,1\} \to \{*\}$ be the function mapping $0 \mapsto *$, $1 \mapsto *$. Then $\{0,1\} \times_f \{0,1\}$ is the set \{(0,0),(0,1),(1,0),(1,1)\}, and $\dhom{1}{(\slice{\catset}{\{*\}})}{f}$ has the following structure:
    \begin{equation*}
    \def\interval{1.75}
        \scalebox{0.75}
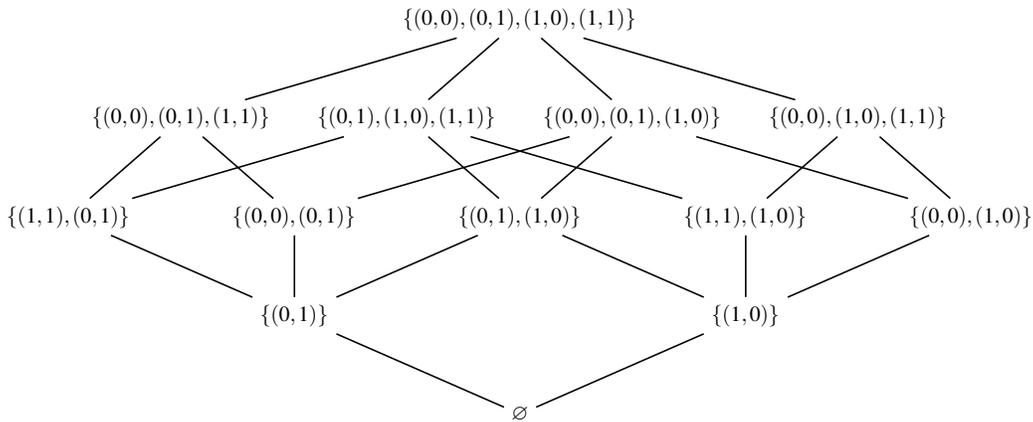
{
        \begin{tikzpicture}
            \node (4a) at (0,4*\interval) {$\{(0,0),(0,1),(1,0),(1,1)\}$};

            \node (3a) at (-6,3*\interval) {$\{(0,0),(0,1),(1,1)\}$};
            \node (3b) at (-2,3*\interval) {$\{(0,1),(1,0),(1,1)\}$};
            \node (3c) at (2,3*\interval) {$\{(0,0),(0,1),(1,0)\}$};
            \node (3d) at (6,3*\interval) {$\{(0,0),(1,0),(1,1)\}$};

            \node (2a) at (-8,2*\interval) {$\{(1,1),(0,1)\}$};
            \node (2b) at (-4,2*\interval) {$\{(0,0),(0,1)\}$};
            \node (2c) at (0,2*\interval) {$\{(0,1),(1,0)\}$};
            \node (2d) at (4,2*\interval) {$\{(1,1),(1,0)\}$};
            \node (2e) at (8,2*\interval) {$\{(0,0),(1,0)\}$};

            \node (1a) at (-4,1*\interval) {$\{(0,1)\}$};
            \node (1b) at (4,1*\interval) {$\{(1,0)\}$};

            \node (0) at (0,0) {$\varnothing$};

            \draw[thick] (3a) -- (4a);
            \draw[thick] (3b) -- (4a);
            \draw[thick] (3c) -- (4a);
            \draw[thick] (3d) -- (4a);
            
            \draw[thick] (2a) -- (3a);
            \draw[thick] (2a) -- (3b);
            \draw[thick] (2b) -- (3a);
            \draw[thick] (2b) -- (3c);
            \draw[thick] (2c) -- (3b);
            \draw[thick] (2c) -- (3c);
            \draw[thick] (2d) -- (3b);
            \draw[thick] (2d) -- (3d);
            \draw[thick] (2e) -- (3d);
            \draw[thick] (2e) -- (3c);

            \draw[thick] (1a) -- (2a);
            \draw[thick] (1a) -- (2b);
            \draw[thick] (1a) -- (2c);
            \draw[thick] (1b) -- (2c);
            \draw[thick] (1b) -- (2d);
            \draw[thick] (1b) -- (2e);

            \draw[thick] (0) -- (1a);
            \draw[thick] (0) -- (1b);
        \end{tikzpicture}
        }
    \end{equation*}
Notice that, via the isomorphism $\Set \simeq \slice{\Set}{\{*\}}$, this is isomorphic to $\dhom{1}{\Set}{\{0, 1\}}$.
\end{example}
To conclude, suppose that two morphisms are both components of the same natural transformation.
Is there a relation between the associated invariants?
The following result answers this question in the affirmative.
\begin{proposition}[Covariance over the domain of a natural transformation] \label{prop: covariance natural transformation}
Let $\fun{F}, \fun{G}\colon \cat{C} \to \cat{D}$ be two functors and let $\alpha\colon \fun{F} \Rightarrow \fun{G}$ be a natural transformation.
For all $i \in \{ 0, 1\}$, the assignment
\begin{equation*}
    x \; \mapsto \; \dhom{i}{(\slice{\cat{D}}{\fun{G}{x}})}{\alpha_x}
\end{equation*}
extends to a functor $\cat{C} \to \pointed{\catpos}$.
\end{proposition}
Notice that this is \emph{not} simply a consequence of \autoref{prop: Homotopy posets are functorial}, that is, it does not arise from the general functoriality result by pre-composition with another functor.\footnote{There is a unifying perspective on the two functoriality results, involving the theory of fibrations and cofibrations of categories; this will be discussed in an extended technical paper.}
It implies that we can naturally map obstructions for $\alpha_x$ to obstructions for $\alpha_y$ along a morphism $f\colon x \to y$ in $\cat{C}$; we can think of morphisms in $\cat{C}$ as inducing a ``flow'' of obstructions to the components of $\alpha$, under which a non-trivial obstruction may be trivialised, but it can never be the case that a non-obstruction is ``un-trivialised''.

    \section{Qualifying compositionality} \label{sec: qualifying}

Now let $\fun{P}\colon \cat{C} \to \cat{D}$ be a \emph{lax} functor of \emph{bicategories}.
This means that, for all triples of objects $X, Y, Z$ in $\cat{C}$, we have two functors
\begin{equation*}
    (\fun{P}-) \Cp (\fun{P}-), \; \fun{P}(- \Cp -)\colon \homset{\cat{C}}{X}{Y} \times \homset{\cat{C}}{Y}{Z} \to \homset{\cat{D}}{\fun{P}X}{\fun{P}Z}
\end{equation*}
connected by a natural transformation, the \emph{laxator} $\varphi\colon (\fun{P}-) \Cp (\fun{P}-) \Rightarrow \fun{P}(- \Cp -)$.\footnote{Technically, the laxators are a family of natural transformations indexed by $X, Y, Z$, but we will leave the indexing implicit.}
As a special case, when $\cat{C}$ and $\cat{D}$ are monoidal categories seen as one-object bicategories, $\fun{P}$ is a lax monoidal functor, and the laxator is a natural transformation $(\fun{P}-) \otimes (\fun{P}-) \Rightarrow \fun{P}(- \otimes -)$.

By Proposition \ref{prop: covariance natural transformation}, we obtain functors $\homset{\cat{C}}{X}{Y} \times \homset{\cat{C}}{Y}{Z} \to \pointed{\catpos}$
sending a pair of morphisms $(f\colon X \to Y, g\colon Y \to Z)$ to the homotopy posets
\begin{equation*}
    \dhom{i}
    {(\slice{\homset{\cat{D}}{\fun{P}X}{\fun{P}Z}}{\fun{P}(f\Cp g)})}
    {\varphi_{f,g}}
\end{equation*}
associated to the component $\varphi_{f,g}$ of the laxator.

In the scenario sketched in the Introduction, the failure of $\varphi_{f,g}$ to be iso is a failure of the ``semantic'' functor $\fun{P}$ to be ``fully compositional'' with respect to the composition $f \Cp g$.
Thus the elements of these homotopy posets may be seen as local \emph{obstructions to compositionality} of $\fun{P}$.
Most interestingly, these obstructions are covariant with respect to the 2-morphisms of $\cat{C}$; thus we can think of ``modifying $f$ and $g$'' by acting on them with a 2-morphism, and see how that affects the obstructions.

\subsection{Open Graphs}\label{subsec: open graphs}
We apply our framework to a couple of tangible examples.
Open graphs, defined in~\cite{Fong2015}, can be thought of as \emph{graphs with interfaces}. Formally, open graphs are (isomorphism classes of) decorated cospans with decorations in the category $\catgrph$ of graphs and homomorphisms. Intuitively, they are depicted as in the examples below, with \emph{input} vertices on the left and \emph{output} vertices on the right:
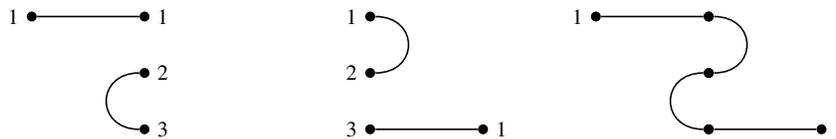
\begin{equation*}
    \scalebox{0.75}{
    \begin{tikzpicture}
        \begin{scope}[xshift=-2cm]
            \node[circle, fill, minimum size=5pt, inner sep=0pt, label=left:{$1$}] (al1) at (-2,0) {};
            \node[circle, fill, minimum size=5pt, inner sep=0pt, label=right:{$1$}] (ar1) at (0,0) {};
            \node[circle, fill, minimum size=5pt, inner sep=0pt, label=right:{$2$}] (ar2) at (0,-1) {};
            \node[circle, fill, minimum size=5pt, inner sep=0pt, label=right:{$3$}] (ar3) at (0,-2) {};
                \draw[thick] (al1) to (ar1);
                \draw[thick, out=180, in=180, looseness=2] (ar2) to (ar3);
        \end{scope}
        \begin{scope}[xshift=2cm]
            \node[circle, fill, minimum size=5pt, inner sep=0pt, label=right:{$1$}] (br3) at (2,-2) {};
            \node[circle, fill, minimum size=5pt, inner sep=0pt, label=left:{$1$}] (bl1) at (0,0) {};
            \node[circle, fill, minimum size=5pt, inner sep=0pt, label=left:{$2$}] (bl2) at (0,-1) {};
            \node[circle, fill, minimum size=5pt, inner sep=0pt, label=left:{$3$}] (bl3) at (0,-2) {};
                \draw[thick, out=0, in=0, looseness=2] (bl1) to (bl2);
                \draw[thick] (bl3) to (br3);
        \end{scope}
        \begin{scope}[xshift=8cm]
            \begin{scope}[xshift=-0cm]
                \node[circle, fill, minimum size=5pt, inner sep=0pt,label=left:{$1$}] (al1) at (-2,0) {};
                \node[circle, fill, minimum size=5pt, inner sep=0pt] (ar1) at (0,0) {};
                \node[circle, fill, minimum size=5pt, inner sep=0pt] (ar2) at (0,-1) {};
                \node[circle, fill, minimum size=5pt, inner sep=0pt] (ar3) at (0,-2) {};
                    \draw[thick] (al1) to (ar1);
                    \draw[thick, out=180, in=180, looseness=2] (ar2) to (ar3);
            \end{scope}
            \begin{scope}[xshift=0cm]
                \node[circle, fill, minimum size=5pt, inner sep=0pt, label=right:{$1$}] (br3) at (2,-2) {};
                \node[circle, fill, minimum size=5pt, inner sep=0pt] (bl1) at (0,0) {};
                \node[circle, fill, minimum size=5pt, inner sep=0pt] (bl2) at (0,-1) {};
                \node[circle, fill, minimum size=5pt, inner sep=0pt] (bl3) at (0,-2) {};
                    \draw[thick, out=0, in=0, looseness=2] (bl1) to (bl2);
                    \draw[thick] (bl3) to (br3);
            \end{scope}
        \end{scope}
    \end{tikzpicture}}
\end{equation*}
Indeed, there is a bicategory $\catopengrph$ that has sets as objects, open graphs as morphisms, and interface-preserving graph homomorphisms as 2-morphisms.
For instance, the first and second open graphs above correspond to morphisms $G\colon \{1\} \to \{1,2,3\}$ and $H\colon \{1,2,3\} \to \{1\}$. 
These morphisms can be composed, resulting in the morphism $G \Cp H\colon \{1\} \to \{1\}$ corresponding to the third open graph in the picture above.

Every graph can be mapped to its \emph{reachability relation}\footnote{Cfr. \cite{lorenz2023causal}, for the similar example of open causal models and causal influence.}: this is a relation on the vertexes of the graph, where two vertexes are considered related iff there is a path between them.
Reachability can be recast as a lax functor $\catopengrph \to \catrel$ to the bicategory of sets, relations, and inclusions of relations, which maps an open graph $G\colon X \to Y$ to the relation $\fun{R}G\colon X \to Y$ defined by
\begin{equation*}
    \text{$\fun{R}G(x, y)$ if and only if there is a path between the input vertex $x$ and the output vertex $y$.}
\end{equation*}
Because $\catrel$ is locally posetal, to define $\fun{R}$ on 2-morphisms it suffices to verify that, if $f\colon G \to G'$ is a graph homomorphism, then $\fun{R}G \subseteq \fun{R}G'$.
The laxators are also uniquely defined.

We can see that this functor is not strong.
In the example above we have that $\fun{R}G \subseteq \{1\} \times \{1,2,3\}$ only contains the pair $(1,1)$, since there are no paths from $1$ to $2$ and from $1$ to $3$ in $G$.
Similarly, $\fun{R}H \subseteq \{1,2,3\} \times \{1\}$ only contains the pair $(3,1)$.
It follows that $\fun{R}G \Cp \fun{R}H\colon \{1\} \to \{1\}$ is the empty relation, but $\fun{R}(G \Cp H)\colon \{1\} \to \{1\}$ is total, so $\fun{R}G \Cp \fun{R}H \subsetneq \fun{R}(G \Cp H)$.

The result is that, if we want to compute the reachability relation of $G \Cp H$ by looking at the reachability relations of $G$ and $H$ separately, we are going to miss something.
This ``compositionality gap'' is tracked by the $\pi_0$ associated to the laxator components $\varphi_{G, H}\colon \fun{R}G \Cp \fun{R}H \subseteq \fun{R}(G \Cp H)$ (because these are all injective, the $\pi_1$ will always be trivial).

In our example, $\dhom{0}{(\slice{\homset{\catrel}{\{1\}}{\{1\}}}{\fun{R}(G\Cp H)})}{\varphi_{G,H}}$ is isomorphic to the poset $(\varnothing < \{(1, 1)\})$ pointed with $\varnothing$, so there is exactly one non-trivial obstruction.
Using covariance, we can think of ``removing the obstruction'' by modifying one or both of the parts $G$ or $H$ with a 2-morphism, that is, with a graph homomorphism.
For example, we can act on $G$ with the homomorphism which identifies the output vertices $1$ and $3$.
The resulting graph $G'$ has $\fun{R}G' = \{(1, 1), (1, 3)\}$, so $\fun{R}G' \Cp \fun{R}H = \fun{R}(G' \Cp H) = \{(1, 1)\}$; correspondingly, we obtain a map of pointed posets from the $\pi_0$ associated to $\varphi_{G, H}$ to the $\pi_0$ associated to $\varphi_{G', H}$, which ``trivialises all obstructions''.
\subsection{Schr\"odinger Compositionality}\label{subsec: schrodinger compositionality}

The name \emph{Schr\"odinger compositionality} was introduced in \cite{coecke2021compositionality} to refer to the form of compositionality that exists in quantum mechanics, where \emph{non-separable states} are present, to disambiguate it from others.
\footnote{For the purposes of this work, we are leaving out of the present analysis the aspects of Schr\"odinger compositionality regarding the  ``ontological interpretation", originally presented in  \cite{coecke2021compositionality}.}
%One key implication of Schr\"odinger compositionality is that ``a state can be more than its parts''.
In the following, we will focus on the special case of a state that can be ``more than its parts''.
This is arguably what makes composition interesting in quantum mechanics: it makes entanglement possible, which Schr\"odinger described as ``the characteristic trait of quantum mechanics'' \cite{Schrodinger_1935}.
In contrast with the example of open graphs, where the ``compositionality gap'' represents an obstacle to a computation strategy, here it can be seen as a positive feature.
Our approach can be used in both contexts; we will focus on the case study of non-separable states, recasting it as the failure of a lax functor to be strong.

In the context of monoidal categories,
%\footnote{Technically, thoughout the section we implicitly assume that our monoidal categories are strict. The example can easily be reworked for general monoidal categories by introducing unitors where it is suitable.}
a \emph{state} is a morphism $\TensorUnit \to A$, where $\TensorUnit$ is the monoidal unit.
We say that a state $\psi\colon \TensorUnit \to A \otimes B$ is \emph{separable} if there exist states $\psi_A\colon \TensorUnit \to A$ and $\psi_B\colon \TensorUnit \to B$ such that $\psi = \psi_A \otimes \psi_B$.%\footnote{Note that we are referring here to the notion of separability of pure states.}, or, graphically:
%
%
%\begin{equation*}
    %\begin{tikzpicture}
     %   \node[draw,thick,minimum width=2cm, minimum height=0.75cm] (phiAB) at (0,0) {$\psi$};
      %      \draw[thick] ($(phiAB.north) - (0.5,0)$) -- +(0,0.75);
       %     \draw[thick] ($(phiAB.north) + (0.5,0)$) -- +(0,0.75);
       % \node (equals) at (1.75,0) {$=$};
       % \node[draw,thick,minimum width=1cm, minimum height=2] (phiA) at (3,0) {$\psi_{A}$};
       %     \draw[thick] (phiA.north) -- +(0,0.75);
       % \node[draw,thick,minimum width=1cm, minimum height=2] (phiB) at (4.5,0) {$\psi_{B}$};
       %     \draw[thick] (phiB.north) -- +(0,0.75);
    %\end{tikzpicture}
%\end{equation*}
%

%
%
\begin{definition}
    Let $(\cat{C}, \otimes, \TensorUnit)$ be a monoidal category.
    The \emph{state functor} of $\cat{C}$ is the representable functor $\homset{\cat{C}}{\TensorUnit}{-}\colon \cat{C} \to \catset$. 
\end{definition}
\begin{proposition}[Laxity of the state functor]\label{prop: state functor lax}
    The state functor lifts to a lax monoidal functor from $(\cat{C}, \otimes, \TensorUnit)$ to $(\catset, \times, \{*\})$, with laxator components
    \begin{align*}
        \varphi_{A,B}\colon \homset{\CategoryC}{\TensorUnit}{A} \times \homset{\CategoryC}{\TensorUnit}{B}
            &\rightarrow 
            \homset{\CategoryC}{\TensorUnit}{A \otimes B}\\
        (\psi_A, \psi_B) 
            &\mapsto
            \psi_A \otimes \psi_B.
    \end{align*}
\end{proposition}
Recall that a monoidal category is \emph{semicartesian} if its monoidal unit is terminal.
The following result is a consequence of the general fact that a functor from a semicartesian to a cartesian monoidal category has a canonical oplax monoidal structure.
\begin{proposition}[Oplaxity of the state functor]\label{prop: state functor oplax}
    Let $(\cat{C}, \otimes, \Term)$ be a semicartesian category.
    Then the state functor lifts to an oplax monoidal functor from $(\cat{C}, \otimes, \Term)$ to $(\catset, \times, \{*\})$.
\end{proposition}
Clearly, there are cases where the state functor is not just lax or oplax, but strong.
The following result captures the well-known fact that in a cartesian monoidal category every state is separable.
\begin{proposition}[Strongness of the state functor]\label{prop: state functor strong}
    If $(\CategoryC, \times, \Term)$ is cartesian, then the state functor is strong monoidal.
\end{proposition}

Having turned Schr\"odinger compositionality into a question about (op)laxity of a functor, we can put our framework to good work.
By \autoref{prop: covariance natural transformation}, we have functors $\cat{C} \times \cat{C} \to \pointed{\catpos}$ sending pairs of objects $(A, B)$ of $\cat{C}$ to the homotopy posets 
\begin{equation} \label{eq: state_dhom }
    \dhom{i}{(\slice{\Set}{\homset{\cat{C}}{I}{A \otimes B}})}{\varphi_{A, B}}, \quad i \in \{ 0, 1 \}.
\end{equation}
Using the description of homotopy posets for slices of $\Set$ from \autoref{sec: obstructions}, we see that
\begin{itemize}
    \item minimal obstructions in $\pi_0$ are in bijection with non-separable states of $A \otimes B$,
    \item minimal obstructions in $\pi_1$ are in bijection with pairs of pairs of states $((\psi_A, \psi_B), (\chi_A, \chi_B))$ such that $\psi_A \otimes \psi_B = \chi_A \otimes \chi_B$.
\end{itemize}
For example, in $(\lcat{Vect}_\mathbb{C}, \otimes, \mathbb{C})$, the monoidal category of complex vector spaces with their tensor product, whenever $A$ and $B$ are at least 2-dimensional, we have instances of both:
\begin{itemize}
    \item the state $1 \mapsto \begin{pmatrix}1 \\ 0\end{pmatrix} \otimes \begin{pmatrix}1 \\ 0\end{pmatrix} + \begin{pmatrix}0 \\ 1\end{pmatrix} \otimes \begin{pmatrix}0 \\ 1\end{pmatrix}$ of $\mathbb{C}^2 \otimes \mathbb{C}^2$ is non-separable,
    \item given any pair of states $(\psi_A, \psi_B)$ and any non-zero $\lambda \in \mathbb{C}$, the pair $(\chi_A, \chi_B) \eqdef (\lambda \psi_A, \invrs{\lambda} \psi_B)$ satisfies $\psi_A \otimes \psi_B = \chi_A \otimes \chi_B$.
\end{itemize}
We can derive a few simple, immediate consequences from the covariance of (\ref{eq: state_dhom }) in the pair $(A, B)$.
\begin{enumerate}
    \item Given morphisms $f\colon A \to A'$, $g\colon B \to B'$, the induced maps of posets preserve the basepoint, that is, map ``non-obstructions'' to ``non-obstructions''.
    In this case, this implies that \emph{it is not possible to entangle a separable state by local actions}, that is, by applying morphisms on $A$ and $B$ separately.
    \item On the other hand, it is, in principle, possible for the induced maps to send non-trivial obstructions to the basepoint.
    For example, in complex vector spaces, acting on $A$ or $B$ with a rank-1 linear map always has a separating effect.
\end{enumerate}

    \section*{Conclusion}

We have introduced our new invariants of categories and stated their fundamental properties, before sketching, through a couple of simple examples, how they may be used to obtain a more fine-grained analysis of ``failures of compositionality'' than a simple yes-or-no judgement.
In an extended technical paper, we will study their formal aspects more in depth, including criteria for the existence of joins and meets, induced monoidal structures, and finer aspects of functoriality.

Most importantly, we hope to have opened a new avenue in ``formal compositionality theory''.
The greatest challenge will be to graduate from proof-of-concept examples to ones that reveal more interesting structure, perhaps in non-$\Set$-like categories where a split epi or mono is not simply a surjective or injective map.
We have been looking at case studies of this sort, which nevertheless have manageable combinatorics permitting an exhaustive study of their homotopy posets, and we hope to discuss them in future work.

	\bibliographystyle{eptcs}
	\bibliography{main}

    \newpage
\section*{Appendix}

\subsection*{Proof of Proposition 10}

    Proving~\autoref{prop: Homotopy posets are functorial} requires to build a hefty amount of theory, which is why we reserve the Appendix for this.

    \begin{definition}[Past extension] \label{dfn:past_extension}
    Let $\cat{A}$ be a category.
    A \emph{past extension of $\cat{A}$} is a functor $\imath\colon \cat{A} \incl \cat{B}$ with the following property: there exists a functor $\indic{\cat{A}}\colon \cat{B} \to \wkarr$ such that
    \begin{equation} \label{eq:past_extension}
    \begin{tikzcd}[sep=scriptsize]
	\cat{A} && \Term \\
	\\
	\cat{B} && \wkarr
	\arrow["{!}", from=1-1, to=1-3]
	\arrow[hook, "\imath", from=1-1, to=3-1]
	\arrow["\indic{\cat{A}}", from=3-1, to=3-3]
	\arrow[hook, "{1}", from=1-3, to=3-3]
	\arrow["\lrcorner"{anchor=center, pos=0.125}, draw=none, from=1-1, to=3-3]
    \end{tikzcd}
    \end{equation}
    is a pullback in $\catcat$.
    \end{definition}
    \begin{remark} \label{rmk:past extension collage}
    The following is an equivalent characterisation of past extensions: there exist a category $\cat{\bar{A}}$ and a profunctor $\fun{H}\colon \opp{\cat{\bar{A}}} \times \cat{A} \to \catset$ such that
    \begin{enumerate}
        \item $\cat{B}$ is isomorphic to the \emph{collage}, also known as \emph{cograph}, of $\fun{H}$, and
        \item $\imath$ is, up to isomorphism, the inclusion of $\cat{A}$ into the collage.
    \end{enumerate}
    A technical name for a functor satisfying the condition on $\imath$ is \emph{codiscrete coopfibration}; it is one leg of a two-sided codiscrete cofibration of categories.

    The idea is that $\imath$ embeds $\cat{A}$ into a larger category, whose objects outside of the image of $\cat{A}$ only have morphisms pointing \emph{towards} $\cat{A}$, hence are ``in the past'' of $\cat{A}$ if we interpret the direction of morphisms as a time direction.
    Notice that the fact that (\ref{eq:past_extension}) is a pullback implies that $\imath$ is injective on objects and morphisms, using their representation as functors from $\Term$ and $\wkarr$, respectively.
    
    The following picture illustrates the bipartition of $\cat{B}$ induced by $\indic{A}$, with the fibre $\cat{\bar{A}}$ of 0 ``in the past'' of the fibre $\cat{A}$ of 1:
    
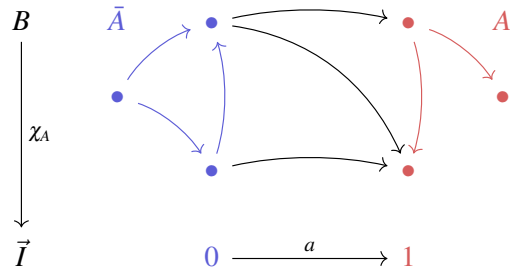
\[\begin{tikzcd}[sep=scriptsize]
	{\cat{B}} & {\blue{\cat{\bar{A}}}} & {\blue{\bullet}} &&& {\red{\bullet}} & {\red{\cat{A}}} \\
	& {\blue{\bullet}} &&&&& {\red{\bullet}} \\
	&& {\blue{\bullet}} &&& {\red{\bullet}} \\
	\wkarr && {\blue{0}} &&& {\red{1}}
	\arrow[color={rgb,255:red,92;green,92;blue,214}, curve={height=-6pt}, from=2-2, to=1-3]
	\arrow[color={rgb,255:red,92;green,92;blue,214}, curve={height=6pt}, from=3-3, to=1-3]
	\arrow[color={rgb,255:red,92;green,92;blue,214}, curve={height=-6pt}, from=2-2, to=3-3]
	\arrow[curve={height=-6pt}, from=1-3, to=1-6]
	\arrow[color={rgb,255:red,214;green,92;blue,92}, curve={height=-6pt}, from=1-6, to=2-7]
	\arrow["{\indic{A}}", from=1-1, to=4-1]
	\arrow[color={rgb,255:red,214;green,92;blue,92}, curve={height=-6pt}, from=1-6, to=3-6]
	\arrow[curve={height=-6pt}, from=3-3, to=3-6]
	\arrow[curve={height=-18pt}, from=1-3, to=3-6]
	\arrow["a", from=4-3, to=4-6]
    \end{tikzcd}\]
    \end{remark}
    \begin{definition}[Category of past extensions] \label{dfn:category_of_past_extensions}
    Let $\cat{A}$ be a category.
    The \emph{category of past extensions of $\cat{A}$} is the large category $\pastext{\cat{A}}$ whose
    \begin{itemize}
        \item objects are past extensions $\imath\colon A \incl B$, and
        \item a morphism from $(\imath\colon A \incl B)$ to $(j\colon A \incl B')$ is a factorisation of $j$ through $\imath$, that is, a functor $\fun{K}\colon \cat{B} \to \cat{B'}$ such that $j = \imath\Cp \fun{K}$.
    \end{itemize}
    \end{definition}
    \begin{proposition}[The indexed category of past extensions of functors]
    Let $\cat{A}$ and $\cat{C}$ be categories.
    Then there exists a functor
    \begin{equation*}
        \extfun{\cat{A}}{\cat{C}}\colon \opp{\pastext{\cat{A}}} \times \cat{C}^\cat{A} \to \catcat
    \end{equation*}
    whose object part is defined as follows: 
    given a past extension $\imath\colon \cat{A} \incl \cat{B}$ and a functor $\fun{F}\colon \cat{A} \to \cat{C}$, the category $\extfun{\cat{A}}{\cat{C}}(\imath, \fun{F})$ is the subcategory of $\cat{C}^\cat{B}$ whose
    \begin{itemize}
        \item objects are (strict) extensions of $\fun{F}$ along $\imath$, that is, functors $\fun{\tilde{F}}\colon \cat{B} \to \cat{C}$ such that
        
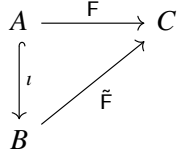
\[\begin{tikzcd}[sep=scriptsize]
	   {\cat{A}} && {\cat{C}} \\
	   \\
	   {\cat{B}}
	   \arrow["\imath", hook, from=1-1, to=3-1]
	   \arrow["{\fun{F}}", from=1-1, to=1-3]
	   \arrow["{\fun{\tilde{F}}}"', from=3-1, to=1-3]
    \end{tikzcd}\]
    strictly commutes, and
    \item morphisms from $\fun{\tilde{F}_1}$ to $\fun{\tilde{F}_2}$ are natural transformations $\tau\colon \fun{\tilde{F}_1} \Rightarrow \fun{\tilde{F}_2}$ that restrict along $\imath$ to the identity natural transformation on $\fun{F}$.
    \end{itemize}
    \end{proposition}
    \begin{proof}
        Given a morphism $\fun{K}\colon (\imath\colon \cat{A} \incl \cat{B}) \to (j\colon \cat{A} \incl \cat{B'})$ in $\pastext{\cat{A}}$,
\begin{equation*}
    \fun{K}^* \eqdef \extfun{\cat{A}}{\cat{C}}(\fun{K},\fun{F})\colon \extfun{\cat{A}}{\cat{C}}(j, \fun{F}) \to \extfun{\cat{A}}{\cat{C}}(\imath, \fun{F})
\end{equation*}
is the functor that acts by precomposition, sending
\begin{itemize}
    \item $\fun{\tilde{F}}\colon \cat{B'} \to \cat{C}$ to $\fun{K}\Cp\fun{\tilde{F}}\colon \cat{B} \to \cat{C}$, and
    \item $\tau\colon \fun{\tilde{F}_1} \Rightarrow \fun{\tilde{F}_2}$ to $\fun{K}\Cp\tau\colon \fun{K}\Cp\fun{\tilde{F}_1} \Rightarrow \fun{K}\Cp\fun{\tilde{F}_2}$.
\end{itemize}
This is well-defined as
\begin{equation*}
    \imath\Cp\fun{K}\Cp\fun{\tilde{F}} = j\Cp\fun{\tilde{F}} = \fun{F}, \quad \quad
    \imath\Cp\fun{K}\Cp\tau = j\Cp\tau = \idd{\fun{F}}.
\end{equation*}
Moreover, it is straightforward to check that
\begin{equation*}
    (\idd{\imath})^* = \idd{\extfun{\cat{A}}{\cat{C}}(\imath, \fun{F})}, \quad \quad 
    (\fun{K}\Cp\fun{L})^* = \fun{L}^*\Cp \fun{K}^*
\end{equation*}
for any composable pair $\fun{K}, \fun{L}$ of morphisms in $\pastext{\cat{A}}$.

Given a natural transformation $\alpha\colon \fun{F} \Rightarrow \fun{G}$ between functors $\fun{F}, \fun{G}\colon \cat{A} \to \cat{C}$, the functor
\begin{equation*}
    \alpha_* \eqdef \extfun{\cat{A}}{\cat{C}}(\imath, \alpha)\colon \extfun{\cat{A}}{\cat{C}}(\imath, \fun{F}) \to \extfun{\cat{A}}{\cat{C}}(\imath, \fun{G})
\end{equation*}
is defined as follows.
Given an object $\fun{\tilde{F}}\colon \cat{B} \to \cat{C}$ of $\extfun{\cat{A}}{\cat{C}}(\imath, \fun{F})$, the functor $\alpha_*\fun{\tilde{F}}\colon \cat{B} \to \cat{C}$ is defined, on each morphism $f\colon x \to y$ in $\cat{B}$, by
\begin{equation*}
    \alpha_*\fun{\tilde{F}}(f) \eqdef
    \begin{cases}
        \fun{G}(f')
            & \text{if $\indic{A}(f) = 1$ and $f = \imath(f')$}, \\
        \fun{\tilde{F}}(f)\Cp\alpha_{y'}
            & \text{if $\indic{A}(f) = a$ and $y = \imath(y')$}, \\
        \fun{\tilde{F}}(f)
            & \text{if $\indic{A}(f) = 0$}.
    \end{cases}
\end{equation*}
By construction $\imath\Cp \alpha_*\fun{\tilde{F}} = \fun{G}$.
The following picture illustrates the definition.
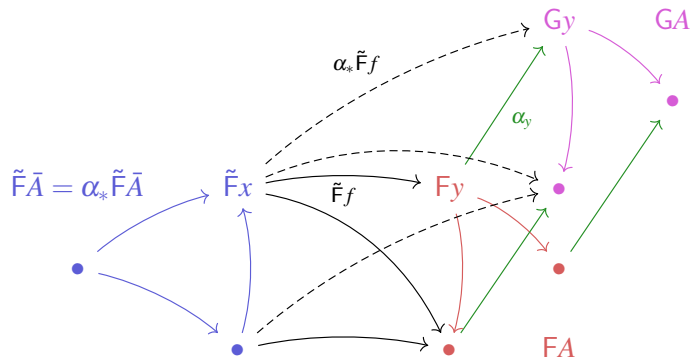
\[\begin{tikzcd}[sep=scriptsize]
	&&&&& {\magenta{\fun{G}y}} & {\magenta{\fun{G}\cat{A}}} \\
	&&&&&& {\magenta{\bullet}} \\
	{\blue{\fun{\tilde{F}}\cat{\bar{A}} = \alpha_*\fun{\tilde{F}}\cat{\bar{A}}}} & {\blue{\fun{\tilde{F}}x}} &&& {\red{\fun{F}y}} & {\magenta{\bullet}} \\
	{\blue{\bullet}} &&&&& {\red{\bullet}} \\
	& {\blue{\bullet}} &&& {\red{\bullet}} & {\red{\fun{F}\cat{A}}}
	\arrow[color={rgb,255:red,92;green,92;blue,214}, curve={height=-6pt}, from=4-1, to=3-2]
	\arrow[color={rgb,255:red,92;green,92;blue,214}, curve={height=6pt}, from=5-2, to=3-2]
	\arrow[color={rgb,255:red,92;green,92;blue,214}, curve={height=-6pt}, from=4-1, to=5-2]
	\arrow["{\fun{\tilde{F}}f}"', curve={height=-6pt}, from=3-2, to=3-5]
	\arrow[color={rgb,255:red,214;green,92;blue,92}, curve={height=-6pt}, from=3-5, to=4-6]
	\arrow[color={rgb,255:red,214;green,92;blue,92}, curve={height=-6pt}, from=3-5, to=5-5]
	\arrow[curve={height=-6pt}, from=5-2, to=5-5]
	\arrow[curve={height=-18pt}, from=3-2, to=5-5]
	\arrow[color={rgb,255:red,214;green,92;blue,214}, curve={height=-6pt}, from=1-6, to=2-7]
	\arrow[color={rgb,255:red,214;green,92;blue,214}, curve={height=-6pt}, from=1-6, to=3-6]
	\arrow["{\alpha_y}"', color={rgb,255:red,36;green,143;blue,36}, from=3-5, to=1-6]
	\arrow[color={rgb,255:red,36;green,143;blue,36}, from=4-6, to=2-7]
	\arrow[color={rgb,255:red,36;green,143;blue,36}, from=5-5, to=3-6]
	\arrow["{\alpha_*\fun{\tilde{F}}f}", curve={height=-12pt}, dashed, from=3-2, to=1-6]
	\arrow[curve={height=-18pt}, dashed, from=3-2, to=3-6]
	\arrow[curve={height=-12pt}, dashed, from=5-2, to=3-6]
\end{tikzcd}\]
Let us show that $\alpha_*\fun{\tilde{F}}$ is well-defined as a functor.
\begin{enumerate}
    \item Given an identity $\idd{x}$ in $\cat{B}$, necessarily $\indic{A}(\idd{x}) = 0$, in which case 
    \[ \alpha_*\fun{\tilde{F}}(\idd{x}) = \fun{\tilde{F}}(\idd{x}) = \idd{\fun{\tilde{F}}(x)}, \] 
    or $\indic{A}(\idd{x}) = 1$, in which case
    \[ \alpha_*\fun{\tilde{F}}(\idd{x}) = \fun{G}(\idd{x'}) = \idd{\fun{G}(x')}, \]
    where $x'$ is the unique lift of $x$ to $\cat{A}$.
    Thus $\alpha_*\fun{\tilde{F}}$ preserves identities.

    \item Given a composable pair $f\colon x \to y$, $g\colon y \to z$, we have the following cases.
    \begin{itemize}
        \item If $\indic{A}(f) = \indic{A}(g) = 1$, then $\indic{A}(f\Cp g) = 1$, and
        \[
            \alpha_*\fun{\tilde{F}}(f)\Cp \alpha_*\fun{\tilde{F}}(g) = \fun{G}(f')\Cp\fun{G}(g') = \fun{G}(f'\Cp g') = \alpha_*\fun{\tilde{F}}(f\Cp g),
        \]
        where $f', g'$ are the unique lifts of $f, g$ to $\cat{A}$.
        \item If $\indic{A}(f) = \indic{A}(g) = 0$, then $\indic{A}(f\Cp g) = 0$, and
        \[
            \alpha_*\fun{\tilde{F}}(f)\Cp  \alpha_*\fun{\tilde{F}}(g) = \fun{\tilde{F}}(f)\Cp \fun{\tilde{F}}(g) = \fun{\tilde{F}}(f\Cp g) = \alpha_*\fun{\tilde{F}}(f\Cp g).
        \]
        \item If $\indic{A}(f) = 0$ and $\indic{A}(g) = a$, then $\indic{A}(f\Cp g) = a$, and
        \[
            \alpha_*\fun{\tilde{F}}(f)\Cp  \alpha_*\fun{\tilde{F}}(g) = \fun{\tilde{F}}(f)\Cp \fun{\tilde{F}}(g)\Cp \alpha_{z'} = \fun{\tilde{F}}(f\Cp g)\Cp \alpha_{z'} = \alpha_*\fun{\tilde{F}}(f\Cp g),
        \]
        where $z'$ is the unique lift of $z$ to $\cat{A}$.
        \item If $\indic{A}(f) = a$ and $\indic{A}(g) = 1$, then $\indic{A}(f\Cp g) = a$, and
        \[
            \alpha_*\fun{\tilde{F}}(f)\Cp  \alpha_*\fun{\tilde{F}}(g) = \fun{\tilde{F}}(f)\Cp \alpha_{y'}\Cp \fun{G}(g') = \fun{\tilde{F}}(f)\Cp \fun{F}(g')\Cp \alpha_{z'},
        \]
        where $g'\colon y' \to z'$ is the unique lift of $g$ to $\cat{A}$, and we used naturality of $\alpha$.
        
        Since $\fun{F}(g') = \tilde{\fun{F}}(\imath(g')) = \tilde{\fun{F}}(g)$, this is equal to 
        \[ \fun{\tilde{F}}(f)\Cp \fun{\tilde{F}}(g)\Cp \alpha_{z'} = \alpha_*\fun{\tilde{F}}(f\Cp g). \]
    \end{itemize}
    No other cases are possible.
\end{enumerate}
This proves that $\alpha_*\fun{\tilde{F}}$ is well-defined.

Given a morphism $\tau\colon \fun{\tilde{F}_1} \Rightarrow \fun{\tilde{F}_2}$ of $\extfun{\cat{A}}{\cat{C}}(\imath, \fun{F})$, the natural transformation $\alpha_*\tau\colon \alpha_*\fun{\tilde{F}_1} \Rightarrow \alpha_*\fun{\tilde{F}_2}$ is defined, on each object $x$ in $\cat{B}$, by
\begin{equation*}
    (\alpha_*\tau)_x \eqdef
    \begin{cases}
        \idd{\fun{G}(x')} & \text{if $\indic{A}(x) = 1$ and $x = \imath(x')$}, \\
        \tau_x & \text{if $\indic{A}(x) = 0$}.
    \end{cases}
\end{equation*}
To show that this is well-defined as a natural transformation, consider a morphism $f\colon x \to y$ in $\cat{B}$.
\begin{itemize}
    \item If $\indic{A}(f) = 1$ and $f'\colon x' \to y'$ is the unique lift of $f$ to $\cat{A}$, then
    \[
        \alpha_*\fun{\tilde{F}_1}(f)\Cp  (\alpha_*\tau)_y =
        \fun{G}(f')\Cp  \idd{\fun{G}(y')} = \idd{\fun{G}(x')}\Cp  \fun{G}(f') = (\alpha_*\tau)_x\Cp  \alpha_*\fun{\tilde{F}_2}(f).
    \]
    \item If $\indic{A}(f) = a$ and $y'$ is the unique lift of $y$ to $\cat{A}$, then
    \[
        \alpha_*\fun{\tilde{F}_1}(f)\Cp  (\alpha_*\tau)_y =
        \fun{\tilde{F}_1}(f)\Cp  \alpha_{y'}\Cp  \idd{\fun{G}(y')} = 
        \fun{\tilde{F}_1}(f)\Cp  \tau_{y}\Cp  \alpha_{y'}
    \]
    since $\tau_y = \tau_{\imath(y')} = \idd{\fun{F}(y')}$.
    By naturality of $\tau$, this is equal to
    \[
        \tau_x\Cp  \fun{\tilde{F}_2}(f)\Cp  \alpha_{y'} =
        (\alpha_*\tau)_x\Cp  \alpha_*\fun{\tilde{F}_2}(f).
    \]
    \item If $\indic{A}(f) = 0$, then
    \[
        \alpha_*\fun{\tilde{F}_1}(f)\Cp  (\alpha_*\tau)_y =
        \fun{\tilde{F}_1}(f)\Cp  \tau_{y} = \tau_x\Cp  \fun{\tilde{F}_2}(f) = (\alpha_*\tau)_x\Cp  \alpha_*\fun{\tilde{F}_2}(f).
    \]
\end{itemize}
This concludes the definition of $\alpha_*$.
It is straightforward to check that
\begin{equation*}
    (\idd{\fun{F}})_* = \idd{\extfun{\cat{A}}{\cat{C}}(\imath, \fun{F})}, \quad \quad 
    (\alpha\Cp \beta)_* = \alpha_* \Cp  \beta_*
\end{equation*}
for all pairs of natural transformations $\alpha, \beta$ composable as morphisms in $\cat{C}^\cat{A}$.
Finally, one can verify that, for all morphisms $\fun{K}\colon \imath \to j$ in $\pastext{\cat{A}}$ and $\alpha\colon \fun{F} \to \fun{G}$ in $\cat{C}^\cat{A}$, the diagram of functors
\[
\begin{tikzcd}[sep=scriptsize]
	\extfun{\cat{A}}{\cat{C}}(j, \fun{F}) && \extfun{\cat{A}}{\cat{C}}(\imath, \fun{F}) \\
	\\
	\extfun{\cat{A}}{\cat{C}}(j, \fun{G}) && \extfun{\cat{A}}{\cat{C}}(\imath, \fun{G})
	\arrow["\fun{K}^*", from=1-1, to=1-3]
	\arrow["\alpha_*", from=1-1, to=3-1]
	\arrow["\fun{K}^*", from=3-1, to=3-3]
	\arrow["\alpha_*", from=1-3, to=3-3]
\end{tikzcd}
\]
commutes in $\catcat$.
Thus we can define $\extfun{\cat{A}}{\cat{C}}(\fun{K}, \alpha)$ as either path in the commutative diagram, and conclude that $\extfun{\cat{A}}{\cat{C}}$ is well-defined as a functor.
    \end{proof}

    \begin{proposition}[Covariance of the $\extfun{\cat{A}}{\cat{C}}$]
    The assignment $\cat{C} \mapsto \extfun{\cat{A}}{\cat{C}}$ extends to a functor
    \begin{equation*}
        \extfun{\cat{A}}{}\colon \catcat \to \laxslice{\lcatcat}{\catcat}.
    \end{equation*}
    \end{proposition}
    \begin{proof}
    Given a functor $\fun{P}\colon \cat{C} \to \cat{D}$, post-composition with $\fun{P}$ defines a functor $\fun{P}_*\colon \cat{C}^\cat{A} \to \cat{D}^\cat{A}$.
    Then there is a natural transformation
    \begin{equation} \label{eq:naturality_extfun}
    \begin{tikzcd}
	{\opp{\pastext{\cat{A}}} \times \cat{C}^\cat{A}} &&& \catcat \\
	\\
	{\opp{\pastext{\cat{A}}} \times \cat{D}^\cat{A}}
	\arrow["{\idd{} \times \fun{P}_*}"', from=1-1, to=3-1]
	\arrow["{\extfun{\cat{A}}{\cat{D}}}"', from=3-1, to=1-4]
	\arrow[""{name=0, anchor=center, inner sep=0}, "{\extfun{\cat{A}}{\cat{C}}}", from=1-1, to=1-4]
	\arrow["{\extfun{\cat{A}}{\fun{P}}}"', shorten <=17pt, shorten >=26pt, Rightarrow, from=0, to=3-1]
    \end{tikzcd}
    \end{equation}
    defined as follows: given a past extension $\imath\colon \cat{A} \incl \cat{B}$ and a functor $\fun{F}\colon \cat{A} \to \cat{C}$, the functor
    \begin{equation*}
        \extfun{\cat{A}}{\fun{P}}(\imath, \fun{F})\colon \extfun{\cat{A}}{\cat{C}}(\imath, \fun{F}) \to \extfun{\cat{A}}{\cat{D}}(\imath, \fun{F}\Cp \fun{P})
    \end{equation*}
    acts both on objects and on morphisms by post-composition with $\fun{P}$.
    It is straightforward to check that the assignment $\fun{P} \mapsto \extfun{\cat{A}}{\fun{P}}$ respects identities and composition in $\catcat$.
    \end{proof}

    \begin{remark}[General functoriality pattern] \label{rmk: functoriality pattern}
    A fixed morphism $\fun{K}$ in $\pastext{\cat{A}}$ is classified by a functor $\wkarr \to \pastext{\cat{A}}$.
    Evaluating $\extfun{\cat{A}}{\cat{C}}$ at $\fun{K}$ thus determines a functor
    \begin{equation*}
        \extfun{\cat{A}}{\cat{C}}(\fun{K}, -)\colon \wkarr \times \cat{C}^\cat{A} \to \catcat,
    \end{equation*}
    which we can curry to obtain a functor
    \begin{equation} \label{eq:general_pattern_cat}
        \Lambda.\extfun{\cat{A}}{\cat{C}}(\fun{K}, -)\colon \cat{C}^\cat{A} \to \catcat^\wkarr.
    \end{equation}
    Given a functor $\fun{P}\colon \cat{C} \to \cat{D}$, we can also ``curry the natural transformation'' in (\ref{eq:naturality_extfun}) to obtain a diagram
    \begin{equation} \label{eq:general_functor_covariance}
    \begin{tikzcd}[sep=large]
	{\cat{C}^\cat{A}} &&& \catcat^\wkarr \\
	\\
	{\cat{D}^\cat{A}}
	\arrow["{\fun{P}_*}"', from=1-1, to=3-1]
	\arrow["{\Lambda.\extfun{\cat{A}}{\cat{D}}(\fun{K},-)}"', from=3-1, to=1-4]
	\arrow[""{name=0, anchor=center, inner sep=0}, "{\Lambda.\extfun{\cat{A}}{\cat{C}}(\fun{K},-)}", from=1-1, to=1-4]
	\arrow["{\Lambda.\extfun{\cat{A}}{\fun{P}}(\fun{K},-)}"{description}, shorten <=17pt, shorten >=26pt, Rightarrow, from=0, to=3-1]
    \end{tikzcd}
    \end{equation}
    which is part of a functor $\catcat \to \laxslice{\lcatcat}{\catcat^\wkarr}$.

    Post-composing with the functor $\catcat^\wkarr \to \pointed{\catpos}$ from (\ref{eq: quotient and posref}) we obtain a covariant family of functors $\cat{C}^\cat{A} \to \pointed{\catpos}$.

    We will show that, for suitable choices of $\cat{A}$ and $\fun{K}$, the image of these functors is included in the subcategory of $\pointed{\catpos}$ on the zeroth and first homotopy posets of $\cat{C}$ or categories associated with $\cat{C}$, exhibiting various kinds of functorial dependence of homotopy posets.
    \end{remark}
    
    \begingroup
    \def\theproposition{\ref{prop: Homotopy posets are functorial}}
    \begin{proposition}[Functoriality of the homotopy posets]
        Let $\cat{C}$ be a category, $i \in \{0, 1\}$.
        Then:
        \begin{enumerate}
            \item the assignment $x \mapsto \dhom{i}{\cat{C}}{x}$ extends to a functor
                $\dhom{i}{\cat{C}}{-}\colon \cat{C} \to \pointed{\catpos}$;
            \item a functor $\fun{F}\colon \cat{C} \to \cat{D}$ induces a natural transformation 
                $\pi_i(\fun{F})\colon \dhom{i}{\cat{C}}{-} \Rightarrow \dhom{i}{\cat{D}}{\fun{F}-}.$
        \end{enumerate}
        Given another functor $\fun{G}\colon \cat{D} \to \cat{E}$,  this assignment satisfies
        \begin{equation*}
            \pi_i(\fun{F}\Cp \fun{G}) = \pi_i(\fun{F}) \Cp \pi_i(\fun{G}), \quad \quad \pi_i(\idd{C}) = \idd{\dhom{i}{\cat{C}}{-}}.
        \end{equation*}
    \end{proposition}
    \addtocounter{proposition}{-1}
    \endgroup
    \begin{proof}
    We will derive the results for both $i \in \{0, 1\}$ from the general functoriality pattern of \autoref{rmk: functoriality pattern}.
    
    First we consider the case $i = 0$.
    Let $\Term$ be the terminal category.
    The inclusion $\fun{K_0}$ of the endpoints of the walking arrow induces a morphism in $\pastext{\Term}$, depicted as follows:
    
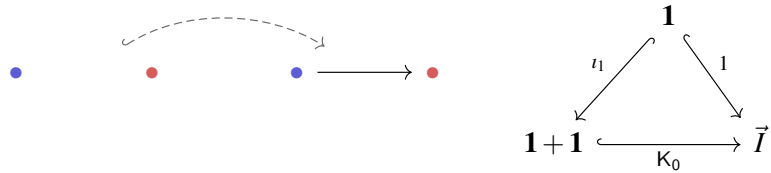
\begin{equation*}
    \begin{tikzcd}[sep=small]
	{\blue{\bullet}} & {} & {\red{\bullet}} &&& {\blue{\bullet}} & {} & {\red{\bullet}}
	\arrow[from=1-6, to=1-8]
	\arrow[color={rgb,255:red,102;green,102;blue,102}, curve={height=-24pt}, shorten <=15pt, shorten >=15pt, dashed, hook, from=1-2, to=1-7]
    \end{tikzcd} \quad \quad
    \begin{tikzcd}[sep=scriptsize]
	& \Term \\
	\\
	{\Term+\Term} && \wkarr
	\arrow["{\fun{K_0}}"', hook, from=3-1, to=3-3]
	\arrow["{\imath_1}"', hook, from=1-2, to=3-1]
	\arrow["1", hook, from=1-2, to=3-3]
    \end{tikzcd}\end{equation*}
    We claim that, up to isomorphism of categories,
    \begin{equation*}
        \Lambda.\extfun{\Term}{\cat{C}}(\fun{K_0}, -)\colon \cat{C}^1 \to \catcat^\wkarr
    \end{equation*}
    sends an object $x$ of $\cat{C}^\Term$ --- which is, equivalently, an object of $\cat{C}$ --- to the slice projection functor
    \begin{equation*}
        \mathrm{dom}\colon \slice{\cat{C}}{x} \to \cat{C}.
    \end{equation*}
    The domain of $\Lambda.\extfun{\Term}{\cat{C}}(\fun{K_0}, x)$ is the category $\extfun{\Term}{\cat{C}}(1, x)$ whose
    \begin{itemize}
        \item objects are functors $f\colon \wkarr \to \cat{C}$ such that 
    \[\begin{tikzcd}[sep=scriptsize]
	   {\Term} && {\cat{C}} \\
	   \\
	   {\wkarr}
	   \arrow["1", hook, from=1-1, to=3-1]
	   \arrow["x", from=1-1, to=1-3]
	   \arrow["f"', from=3-1, to=1-3]
    \end{tikzcd}\]
        commutes, which are in bijection with morphisms $f$ of $\cat{C}$ whose codomain is $x$, and
        \item morphisms from $f$ to $g$ are natural transformations $h\colon f \Rightarrow g$ --- which are in bijection with commutative squares
        \[\begin{tikzcd}[sep=scriptsize]
	   y && z \\
	   \\
	   x && x
	   \arrow["f"', from=1-1, to=3-1]
	   \arrow["{h_1}"', from=3-1, to=3-3]
	   \arrow["{h_0}", from=1-1, to=1-3]
	   \arrow["g", from=1-3, to=3-3]
        \end{tikzcd}\]
        in $\cat{C}$ --- that restrict to the identity along $1\colon \Term \incl \wkarr$, that is, are such that $h_1 = \idd{x}$.
        These are in bijection with factorisations of $f$ through $g$.
    \end{itemize}
    This establishes an isomorphism between $\extfun{\Term}{\cat{C}}(1, x)$ and $\slice{\cat{C}}{x}$.
    The codomain of $\Lambda.\extfun{\Term}{\cat{C}}(\fun{K_0}, x)$ is the category $\extfun{\Term}{\cat{C}}(\imath_1, x)$ whose
    \begin{itemize}
        \item objects are functors $(y_0, y_1)\colon \Term + \Term \to \cat{C}$ such that 
    \[\begin{tikzcd}[sep=scriptsize]
	   {\Term} && {\cat{C}} \\
	   \\
	   {\Term + \Term}
	   \arrow["{\imath_1}", hook, from=1-1, to=3-1]
	   \arrow["x", from=1-1, to=1-3]
	   \arrow["{(y_0, y_1)}"', from=3-1, to=1-3]
    \end{tikzcd}\]
        commutes, which are in bijection with pairs of objects $(y_0, y_1)$ of $\cat{C}$ such that $y_1 = x$, which are in bijection with objects of $\cat{C}$, and
        \item morphisms from $(y, x)$ to $(z, x)$ are in bijection with pairs of morphisms
    \[\begin{tikzcd}[sep=scriptsize]
	y && z \\
	x && x
	\arrow["{h_1}"', from=2-1, to=2-3]
	\arrow["{h_0}", from=1-1, to=1-3]
    \end{tikzcd}\]
        in $\cat{C}$ that restrict to the identity along $\imath_1$, that is, are such that $h_1 = \idd{x}$.
        These are in bijection with morphisms $y \to z$.
    \end{itemize}
    This establishes an isomorphism between $\extfun{\Term}{\cat{C}}(\imath_1, x)$ and $\cat{C}$.
    The functor $\extfun{\Term}{\cat{C}}(\fun{K_0}, x)$ acts by restriction of $f\colon \wkarr \to \cat{C}$ along $\fun{K_0}\colon \Term+\Term \incl \wkarr$; through the isomorphisms, this acts by mapping $f\colon y \to x$ to its domain $y$.
    This is, by inspection, the same as the action of $\mathrm{dom}$.

    We define
    \begin{equation*}
        \dhom{0}{\cat{C}}{-}\colon \cat{C} \to \pointed{\catpos}
    \end{equation*}
    to be the post-composition of $\Lambda.\extfun{\Term}{\cat{C}}(\fun{K_0}, -)$ with the functor of \autoref{eq: quotient and posref}.
    It follows from our argument that, up to isomorphism, this sends $x$ to the homotopy poset $\dhom{0}{\cat{C}}{x}$.
    The covariance in $\cat{C}$ then follows as an instance of \autoref{eq:general_functor_covariance}: given a functor $\fun{F}\colon \cat{C} \to \cat{D}$, we whisker the natural transformation $\Lambda.\extfun{\Term}{\fun{F}}(\fun{K_0}, -)$ with the functor of (\ref{eq: quotient and posref}) to obtain $\pi_0(\fun{F})\colon \dhom{0}{\cat{C}}{-} \Rightarrow \dhom{0}{\cat{D}}{\fun{F}-}$.

    Now, let us focus on the first homotopy poset.
    The functor $\fun{K_1}$ identifying two parallel arrows also induces a morphism in $\pastext{\Term}$, depicted as follows:
    
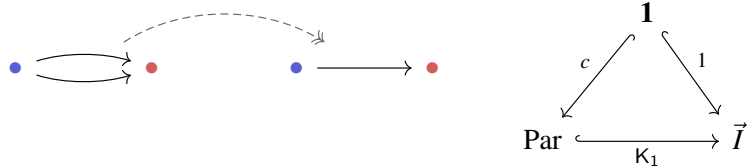
\begin{equation*}
    \begin{tikzcd}[sep=small]
	{\blue{\bullet}} & {} & {\red{\bullet}} &&& {\blue{\bullet}} & {} & {\red{\bullet}}
	\arrow[curve={height=-6pt}, from=1-1, to=1-3]
	\arrow[curve={height=6pt}, from=1-1, to=1-3]
	\arrow[from=1-6, to=1-8]
	\arrow[color={rgb,255:red,102;green,102;blue,102}, curve={height=-24pt}, shorten <=15pt, shorten >=15pt, dashed, two heads, from=1-2, to=1-7]
    \end{tikzcd}
    \quad \quad
    \begin{tikzcd}[sep=scriptsize]
	& \Term \\
	\\
	  \mathrm{Par} && \wkarr
	\arrow["{\fun{K_1}}"', hook, from=3-1, to=3-3]
	\arrow["c"', hook, from=1-2, to=3-1]
	\arrow["1", hook, from=1-2, to=3-3]
    \end{tikzcd}
    \end{equation*}
    Here, $\mathrm{Par}$ denotes the ``walking parallel pair of arrows''.
    We claim that, up to isomorphism of categories,
    \begin{equation*}
        \Lambda.\extfun{\Term}{\cat{C}}(\fun{K_1}, -)\colon \cat{C} \to \catcat^\wkarr
    \end{equation*}
    sends an object $x$ of $\cat{C}$ to the slice projection functor
    \begin{equation*}
        \mathrm{dom}\colon \slice{\pararr{\cat{C}}{x}}{(\idd{x}, \idd{x})} \to \pararr{\cat{C}}{x}.
    \end{equation*}
    We have already established that the domain of $\Lambda.\extfun{\Term}{\cat{C}}(\fun{K_1}, x)$, which is the category $\extfun{\Term}{\cat{C}}(1, x)$, is isomorphic to $\slice{\cat{C}}{x}$, which can be shown to be isomorphic to $\slice{\pararr{\cat{C}}{x}}{(\idd{x}, \idd{x})}$ using Proposition \ref{prop: subterminal as weak terminal parallel arrow}.

    The codomain of $\Lambda.\extfun{\Term}{\cat{C}}(\fun{K_1}, x)$ is the category $\extfun{\Term}{\cat{C}}(c, x)$ whose
    \begin{itemize}
        \item objects are functors $(f_0, f_1)\colon \mathrm{Par} \to \cat{C}$ such that 
    \[\begin{tikzcd}[sep=scriptsize]
	{\Term} && {\cat{C}} \\
	\\
	{{\mathrm{Par}}}
	\arrow["c", hook, from=1-1, to=3-1]
	\arrow["x", from=1-1, to=1-3]
	\arrow["{(f_0, f_1)}"', from=3-1, to=1-3]
    \end{tikzcd}\]
    commutes, which are in bijection with pairs of morphisms $(f_0, f_1)$ of $\cat{C}$ whose codomain is $x$, and
    \item morphisms from the pair $(f_0, f_1)$ to $(g_0, g_1)$ are natural transformations $h\colon (f_0, f_1) \Rightarrow (g_0, g_1)$ that restrict to the identity along $c$, which are in bijection with morphisms $h$ such that $f_0 = h;g_0$ and $f_1 = h;g_1$.
    \end{itemize}
    This establishes an isomorphism between $\extfun{\Term}{\cat{C}}(c, x)$ and $\pararr{\cat{C}}{x}$.

    The functor $\extfun{\Term}{\cat{C}}(\fun{K_1}, x)$ acts by precomposing $f\colon \wkarr \to \cat{C}$ with $\fun{K_1}\colon \mathrm{Par} \to \wkarr$, which through the isomorphisms sends a pair $(f, f)$ with its unique morphism to $(\idd{x}, \idd{x})$ to the pair $(f, f)$ on its own.
    This is, by inspection, the same as the action of $\mathrm{dom}$.

    We define
    \begin{equation*}
        \dhom{1}{\cat{C}}{-}\colon \cat{C} \to \pointed{\catpos}
    \end{equation*}
    to be the post-composition of $\Lambda.\extfun{\Term}{\cat{C}}(\fun{K_1}, -)$ with the functor of \autoref{eq: quotient and posref}.
    It follows from our argument that, up to isomorphism, this sends $x$ to the homotopy poset $\dhom{1}{\cat{C}}{x}$.
    Again, we obtain covariance in $\cat{C}$ by whiskering instances of \autoref{eq:general_functor_covariance}.
    This completes the proof.
    \end{proof}

\end{document}